\documentclass[11pt,reqno]{amsart}
\usepackage[english]{babel}
\usepackage[latin1]{inputenc}
\usepackage{amsmath,amsfonts,amssymb,amsthm,amscd,array,mathrsfs}
\usepackage{pstricks}
 \usepackage[all]{xy}
\usepackage{graphicx}
\usepackage{color}              
\usepackage{epsfig}
\usepackage{epstopdf}
\usepackage{textcomp}
\theoremstyle{plain}
\newtheorem{thm}{Theorem}
\newtheorem{lem}{Lemma}[section]
\newtheorem{cor}[lem]{Corollary}
\newtheorem{prop}[lem]{Proposition}
\theoremstyle{definition}
\newtheorem{defi}[lem]{Definition}
\newtheorem{rem}[lem]{Remark}
\newtheorem{ex}[lem]{Example}

\newcommand{\R}{\mathbb{R}}
\newcommand{\Z}{\mathbb{Z}}
\newcommand{\N}{\mathbb{N}}
\newcommand{\C}{\mathbb{C}}

\newcommand{\qH}{\mathbb{H}}

\newcommand{\K}{\mathbb{K}}
\newcommand{\I}{\mathbb{I}}

\newcommand{\X}{\mathcal{X}}
\newcommand{\Y}{\mathcal{Y}}

\newcommand{\G}{\mathcal{G}}

\newcommand{\E}{\mathcal{E}}

\newcommand*{\EnsQuot}[2]%
{\ensuremath{%
    #1/\!\raisebox{-.65ex}{\ensuremath{{#2}}}}}

\newcommand{\Top}{\rule{0pt}{3ex}}

\mathchardef\za="710B  
\mathchardef\zb="710C  
\mathchardef\zg="710D  
\mathchardef\zd="710E  
\mathchardef\zve="710F 
\mathchardef\zz="7110  
\mathchardef\zh="7111  
\mathchardef\zvy="7112 
\mathchardef\zi="7113  
\mathchardef\zk="7114  
\mathchardef\zl="7115  
\mathchardef\zm="7116  
\mathchardef\zn="7117  
\mathchardef\zx="7118  
\mathchardef\zp="7119  
\mathchardef\zr="711A  
\mathchardef\zs="711B  
\mathchardef\zt="711C  
\mathchardef\zu="711D  
\mathchardef\phi="711E 
\mathchardef\zq="711F  
\mathchardef\zc="7120  
\mathchardef\zw="7121  
\mathchardef\ze="7122  
\mathchardef\zy="7123  
\mathchardef\zf="7124  
\mathchardef\zvr="7125 
\mathchardef\zvs="7126 
\mathchardef\zf="7127  
\mathchardef\zG="7000  
\mathchardef\zD="7001  
\mathchardef\zY="7002  
\mathchardef\zL="7003  
\mathchardef\zX="7004  
\mathchardef\zP="7005  
\mathchardef\zS="7006  
\mathchardef\zU="7007  
\mathchardef\zF="7008  
\mathchardef\zW="700A  

\newcommand{\degr}{\widetilde}

\DeclareMathOperator{\tr}{tr}
\DeclareMathOperator{\gtr}{\G\!\tr}

\DeclareMathOperator{\gdet}{\G\!\!\det}
\DeclareMathOperator{\gdetinv}{\G\!\!\det\,\!\!^{-1}}
\DeclareMathOperator{\Ddet}{D\!\det}
\DeclareMathOperator{\ber}{Ber}
\DeclareMathOperator{\gber}{\G\!\ber}
\DeclareMathOperator{\qi}{i}
\DeclareMathOperator{\qj}{j}
\DeclareMathOperator{\qk}{k}
\newcommand{\op}[1]{\!\!\mathop{\rm ~#1}\nolimits}
\newcommand{\be}{\begin{equation}}
\newcommand{\ee}{\end{equation}}
\newcommand{\bea}{\begin{eqnarray}}
\newcommand{\eea}{\end{eqnarray}}
\newcommand{\beas}{\begin{eqnarray*}}
\newcommand{\eeas}{\end{eqnarray*}}
\newcommand{\la}{\langle}
\newcommand{\ra}{\rangle}
\newcommand{\lp}{\left(}
\newcommand{\rp}{\right)}
\newcommand{\lpq}{\left[}
\newcommand{\rpq}{\right]}

\def\g{\gamma}
\def\G{\Gamma}

\font\black=cmbx10 \font\sblack=cmbx7 \font\ssblack=cmbx5 \font\blackital=cmmib10  \skewchar\blackital='177
\font\sblackital=cmmib7 \skewchar\sblackital='177 \font\ssblackital=cmmib5 \skewchar\ssblackital='177
\font\sanss=cmss12 \font\ssanss=cmss8 scaled 900 \font\sssanss=cmss8 scaled 600 \font\blackboard=msbm10
\font\sblackboard=msbm7 \font\ssblackboard=msbm5 \font\caligr=eusm10 \font\scaligr=eusm7 \font\sscaligr=eusm5
    
\font\bsymb=cmsy10 scaled\magstep2
\def\all#1{\setbox0=\hbox{\lower1.5pt\hbox{\bsymb
       \char"38}}\setbox1=\hbox{$_{#1}$} \box0\lower2pt\box1\;}
\def\exi#1{\setbox0=\hbox{\lower1.5pt\hbox{\bsymb \char"39}}
       \setbox1=\hbox{$_{#1}$} \box0\lower2pt\box1\;}

\newfam\bifam
\textfont\bifam=\blackital \scriptfont\bifam=\sblackital \scriptscriptfont\bifam=\ssblackital

\newfam\blfam
\textfont\blfam=\black \scriptfont\blfam=\sblack \scriptscriptfont\blfam=\ssblack

\newfam\bbfam
\textfont\bbfam=\blackboard \scriptfont\bbfam=\sblackboard \scriptscriptfont\bbfam=\ssblackboard

\newfam\ssfam
\textfont\ssfam=\sanss \scriptfont\ssfam=\ssanss \scriptscriptfont\ssfam=\sssanss
\def\sss#1{{\fam\ssfam\relax#1}}
\newfam\clfam
\textfont\clfam=\caligr \scriptfont\clfam=\scaligr \scriptscriptfont\clfam=\sscaligr

\def\pmb#1{\setbox0\hbox{${#1}$} \copy0 \kern-\wd0 \kern.2pt \box0}
\def\pmbb#1{\setbox0\hbox{${#1}$} \copy0 \kern-\wd0
      \kern.2pt \copy0 \kern-\wd0 \kern.2pt \box0}
\def\pmbbb#1{\setbox0\hbox{${#1}$} \copy0 \kern-\wd0
      \kern.2pt \copy0 \kern-\wd0 \kern.2pt
    \copy0 \kern-\wd0 \kern.2pt \box0}
\def\pmxb#1{\setbox0\hbox{${#1}$} \copy0 \kern-\wd0
      \kern.2pt \copy0 \kern-\wd0 \kern.2pt
      \copy0 \kern-\wd0 \kern.2pt \copy0 \kern-\wd0 \kern.2pt \box0}
\def\pmxbb#1{\setbox0\hbox{${#1}$} \copy0 \kern-\wd0 \kern.2pt
      \copy0 \kern-\wd0 \kern.2pt
      \copy0 \kern-\wd0 \kern.2pt \copy0 \kern-\wd0 \kern.2pt
      \copy0 \kern-\wd0 \kern.2pt \box0}

\begin{document}

\title[$(\Z_2)^n$-graded Trace and Berezinian]
{Higher Trace and Berezinian\\\vskip0.2cm of\\\vskip0.2cm Matrices
over a Clifford Algebra}

\author{Tiffany Covolo
\hskip 1cm Valentin Ovsienko \hskip 1cm Norbert Poncin}
\address{
Tiffany Covolo, Universit\'e du Luxembourg, UR en math\'ematiques,
6, rue Richard Coudenhove-Kalergi, L-1359 Luxembourg,
Grand-Duch\'e de Luxembourg}
\address{
Valentin Ovsienko, CNRS, Institut Camille Jordan, Universit\'e
Claude Bernard Lyon~1, 43, boulevard du 11 novembre 1918, F-69622
Villeurbanne cedex, France}
\address{
Norbert Poncin, Universit\'e du Luxembourg, UR en math\'ematiques,
6, rue Richard Coudenhove-Kalergi, L-1359 Luxembourg,
Grand-Duch\'e de Luxembourg}\email{tiffany.covolo@uni.lu,
ovsienko@math.univ-lyon1.fr, norbert.poncin@uni.lu}

\begin{abstract}
We define the notions of trace, determinant and, more generally,
Berezinian of matrices over a $(\Z_2)^n$-graded commutative
associative algebra~$A$. The applications include a new approach
to the classical theory of matrices with coefficients in a
Clifford algebra, in particular of quaternionic matrices. In a
special case, we recover the classical Dieudonn\'e determinant of
quaternionic matrices, but in general our quaternionic determinant
is different. We show that the graded determinant of purely even
$(\Z_2)^n$-graded matrices of degree~$0$ is polynomial in its
entries. In the case of the algebra $A=\qH$ of quaternions, we
calculate the formula for the Berezinian in terms of a product of
quasiminors in the sense of Gelfand, Retakh, and Wilson. The
graded trace is related to the graded Berezinian (and determinant)
by a $(\Z_2)^n$-graded version of Liouville's formula.
\end{abstract}

\vspace{2mm}
\noindent \subjclass[2010]{17A70, 58J52, 58A50, 15A66,
11R52}\medskip \noindent\keywords{Clifford linear algebra,
quaternionic determinants, $(\Z_2)^n$-graded commutative algebra}
\maketitle \thispagestyle{empty}

\tableofcontents

\section{Introduction}

Linear algebra over quaternions is a classical subject. Initiated
by Hamilton and Cayley, it was further developed by Study
\cite{Stu} and Dieudonn\'e \cite{Die}, see \cite{Asl} for a
survey. The best known version of quaternionic determinant is due
to Dieudonn\'e, it is far of being elementary and still attracts a
considerable interest, see \cite{GRW03}. The Dieudonn\'e
determinant is not related to any notion of trace. To the best of
our knowledge, the concept of trace is missing in the existing
theories of quaternionic matrices.
\medskip

The main difficulty of any theory of matrices over quaternions,
and more generally over Clifford algebras, is related to the fact
that these algebras are not commutative. It turns out however,
that the classical algebra $\qH$ of quaternions can be understood
as a graded-commutative algebra. It was shown in \cite{Lyc},
\cite{AM}, \cite{AM1} that $\qH$ is a graded commutative algebra
over the Abelian group $(\Z_2)^2=\Z_2\times\Z_2$ (or over the even
part of $(\Z_2)^3$, see \cite{MO}). Quite similarly, every
Clifford algebra with $n$ generators is $(\Z_2)^n$-graded
commutative \cite{AM1} (furthermore, a Clifford algebra is
understood as even $(\Z_2)^{n+1}$-graded commutative algebra in
\cite{OM}). This viewpoint suggests a natural approach to linear
algebra over Clifford algebras as generalized Superalgebra.
\medskip

Geometric motivations to consider $(\Z_2)^n$-gradings come from
the study of higher vector bundles \cite{GR07}. If $E$ denotes a
vector bundle with coordinates $(x,\xi)$, a kind of universal
Legendre transform
$$T^*E\ni (x,\xi,y,\eta)\leftrightarrow (x,\eta,y,-\xi)\in
T^*E^*$$ provides a natural and rich $(\Z_2)^2$-degree $((0,0),
(1,0), (1,1), (0,1))$ on $T^*[1]E[1]$. Multigraded vector bundles
give prototypical examples of $(\Z_2)^n$-graded manifolds.
\medskip

Quite a number of geometric structures can be encoded in
supercommutative algebraic structures, see e.g., \cite{GKP11b},
\cite{GKP10a}, \cite{GKP10b}, \cite{GP04}. On the other hand,
supercommutative algebras define supercommutative geometric
spaces. It turns out, however, that the classical $\Z_2$-graded
commutative algebras $\op{Sec}(\wedge E^*)$ of vector bundle forms
are far from being sufficient. For instance, whereas Lie
algebroids are in 1-to-1 correspondence with homological vector
fields of split supermanifolds $\op{Sec}(\wedge E^*)$, the
supergeometric interpretation of Loday algebroids \cite{GKP11a}
requires a $\Z_2$-graded commutative algebra of non-Grassmannian
type, namely the shuffle algebra $\mathcal{D}(E)$ of specific
multidifferential operators. However, not only other types of
algebras, but also more general grading groups must be
considered.\medskip

Let us also mention that classical Supersymmetry and
Supermathematics are not completely sufficient for modern physics
(i.e., the description of anyons, paraparticles).
\medskip

All the aforementioned problems are parts of our incentive to
investigate the basic notions of linear algebra over
a~$(\Z_2)^n$-graded commutative unital associative algebra~$A$. We
consider the space $\sss{M}(\mathbf{r};A)$ of matrices with
coefficients in $A$ and introduce the notions of graded trace and
Berezinian (in the simplest case of purely even matrices we will
talk of the determinant). We prove an analog of the Liouville
formula that connects both concepts. Although most of the results
are formulated and proved for arbitrary $A$, our main goal is to
develop a new theory of matrices over Clifford algebras and, more
particularly, over quaternions.
\medskip

Our main results are as follows:\medskip

\begin{itemize}
\item There exists a unique homomorphism of graded $A$-modules and
graded Lie algebras
$$\op{\zG tr}: \sss{M}(\mathbf{r};A)\to A\;,$$
defined for arbitrary matrices with coefficients in $A$.\smallskip
\item There exists a unique map
$$\gdet:\sss{M}^0(\mathbf{r}_0;A)\to A^0\;,$$
defined on purely even homogeneous matrices of degree 0 with
values in the commutative subalgebra $A^0\subset{}A$ consisting of
elements of degree 0 and characterized by three properties: a)
$\gdet$ is multiplicative, b) for a block-diagonal matrix $\gdet$
is the product of the determinants of the blocks, c) $\gdet$ of a
lower (upper) unitriangular matrix equals 1. In the case $A=\qH$,
the absolute value of $\gdet$ coincides with the classical
Dieudonn\'e determinant.\smallskip \item There exists a unique
group homomorphism
$$\gber:\sss{GL}^0(\mathbf{r};A)\to (A^0)^{\times}\;,$$
defined on the group of invertible homogeneous matrices of degree
$0$ with values in the group of invertible elements of $A^0$,
characterized by properties similar to a), b), c).\smallskip \item
The graded Berezinian is connected with the graded trace by a
$(\Z_2)^n$-graded version of Liouville's formula
$$\gber(\exp(\ze X))=\exp(\op{\zG tr}(\ze X))\;,$$ where $\ze$
is a nilpotent formal parameter of degree 0 and $X$ a graded
matrix.\smallskip \item For the matrices with coefficients in a
Clifford algebra, there exists a unique way to extend the graded
determinant to homogeneous matrices of degree different from zero,
if and only if the total matrix dimension$\,|\,\mathbf{r}|$
satisfies the condition
$$|\,\mathbf{r}|=0,1\quad(\op{mod} 4).$$
In the case of matrices over $\qH$, this graded determinant
differs from that of Dieudonn\'e.
\end{itemize}

\medskip

The reader who wishes to gain a quick and straightforward insight
into some aspects of the preceding results, might envisage having
a look at Section \ref{ExSect} at the end of this paper, which can
be read independently.\medskip

Our main tools that provide most of the existence results and
explicit formul{\ae} of graded determinants and graded
Berezinians, are the concepts of quasideterminants and
quasiminors, see~\cite{GGRW05} and references therein.
\medskip

Let us also mention that in the case of matrices over a Clifford
algebra, the restriction for the dimension of the $A$-module,
$|\,\mathbf{r}|=0,1$ (mod 4), provides new insight into the old
problem initiated by Arthur Cayley, who considered specifically
two-dimensional linear algebra over quaternions. It follows that
Cayley's problem has no solution, at least within the framework of
graded algebra adopted in this paper. In particular, the notion of
determinant of a quaternionic $(2\times2)$-matrix related to a
natural notion of trace does not exist.
\medskip

Basic concepts of $(\Z_2)^n$-graded Geometry based on the linear
algebra developed in the present paper are being studied in a
separate work; applications to quaternionic functions and to
Mathematical Physics are expected. We also hope to investigate the
cohomological nature of $(\Z_2)^n$-graded Berezinians, as well as
the properties of the characteristic polynomial, see \cite{KV} and
references therein.
\medskip

 Another method to treat the problem of generalizing superalgebras and related notions, alternative to the one presented in this paper and which makes use of category theory, is being studied in a separate work. This approach follows from results by Scheunert in \cite{Sch} (in the Lie algebra setting) and Nekludova (in the commutative algebra setting).
 An explicit description of the results of the latter first appeared in \cite{ptrzvk}, and can also be found in \cite{SoS}.

\section{$(\Z_2)^n$-Graded Algebra}\label{ModSect}

In this section we fix terminology and notation used throughout this paper.
Most of the definitions extend well-known definitions of usual superlagebra \cite{Lei80},
see also \cite{VAR}.\medskip

\subsection{General Notions}

Let $(\zG,+)$ be an Abelian group endowed
with a symmetric bi-additive map
$$\la \;,\;\ra:\zG\times\zG\to\Z_2\,.$$
That is,
$$\la\g,\g'\ra=\la\g',\g\ra
\qquad\hbox{and} \qquad
\la\g+\g',\g''\ra=\la\g,\g''\ra+\la\g',\g''\ra \,.
$$
The \textit{even} subgroup $\G_0$ consists of elements $\g\in\G$
such that $\la\g,\g\ra=0$.
One then has a splitting
$$
\G=\G_0\,\cup\,\G_1 \,,
$$
where $\G_1$ consists of \textit{odd} elements $\g\in\G$
such that $\la\g,\g\ra=1$.
Of course, $\G_1$ is not a subgroup of $\G$. \medskip

A basic example is the additive group
$(\Z_2)^n$, $n\in \mathbb{N}$, equipped with the standard scalar product $\la \;,\;\ra$
of $n$-vectors, defined over $\Z_2$,  Section \ref{CliS}.\medskip

A {\it graded vector space} is a direct sum
$$V=\bigoplus_{\zg\in\zG}V^{\zg}$$
of vector spaces $V^{\zg}$ over a
commutative field $\K$ (that we will always assume of characteristic 0).
A graded vector space is always a direct sum:
$$
V=V_0\oplus{}V_1
$$
of its even subspace $V_0=\bigoplus_{\zg\in\zG_0}V^{\zg}$ and its odd
subspace $V_1=\bigoplus_{\zg\in\zG_1}V^{\zg}$.
\medskip

If $V$ and $W$ are
graded vector spaces, then one has:
$$
\sss{Hom}_{\K}(V,W)=\bigoplus_{\zg\in\zG}\sss{Hom}_{\K}^{\zg}(V,W) \,,
$$
where $\sss{Hom}^{\zg}_{\K}(V,W)$
(or simply $\sss{Hom}^{\zg}(V,W)$) the vector space of
$\K$-linear maps  of weight $\zg$
$$
\ell:V\to W,
\qquad
\ell(V^{\zd})\subset W^{\zd+\zg}\,.
$$
We also use the standard notation $\sss{End}_{\K}(V):=\sss{Hom}_{\K}(V,V)$.
\medskip

A {\it graded algebra} is an algebra $A$ which has a structure of a graded vector space $A$
such that the operation of multiplication respects the grading:
$$A^{\zg}A^{\zd}\subset A^{\zg+\zd}\,.$$
If $A$ is associative (resp., associative and unital), we
call it a {\it graded associative algebra} (resp., {\it graded
associative unital algebra}). In this case, the operation of multiplication is denoted by ``$\cdot$''.
A graded associative algebra $A$ is called
\textit{graded commutative} if, for any homogeneous elements $a,b\in A$, we have
\begin{equation}
\label{GCon}
b\cdot a=(-1)^{\la \tilde a,\tilde b\ra}a\cdot b\,.
\end{equation}
Here and below $\tilde a\in\G$ stands for the degree of $a$, that is $a\in A^{\tilde a}$.
Note that graded commutative algebras are also known in the literature
under the name of ``color commutative'' algebras. \medskip

Our main examples of graded commutative algebras
are the classical Clifford algebras equipped with $(\Z_2)^n$-grading,
see Section \ref{CliS}. \medskip

A graded algebra $A$ is called a  {\it graded Lie algebra}
if it is graded anticommutative and satisfies the graded Jacobi
identity.
The operation of multiplication is then denoted by $[\;,\;]$.
The identities read explicitly:
 $$
 \begin{array}{rcl}
 [a,b]&=&-(-1)^{\la \tilde a,\tilde b\ra}[b,a]\,,\\[4pt]
 [a,[b,c]]&=&[[a,b],c]+(-1)^{\la \tilde a,\tilde
b\ra}[b,[a,c]]\,.
\end{array}$$
Graded Lie algebras are often called ``color Lie algebras'', see
\cite{Sch}.
The main example of a graded Lie algebra is the space $\sss{End}_{\K}(V)$
equipped with the commutator:
\begin{equation}
\label{GCom}
[X,Y]=X\circ{}Y - (-1)^{\la\widetilde{X},\widetilde{Y}\ra}\,Y\circ{}X\,,
\end{equation}
for homogeneous $X,Y\in\sss{End}_{\K}(V)$ and extended by linearity.
\medskip

A graded vector space $M$ is called a {\it graded (left) module}
over a graded commutative algebra~$A$ if
there is a $\K$-linear
map $\zl:A\to \sss{End}_{\K}(M)$ of weight $0\in\zG$ that
satisfies
$$\zl(a)\circ\zl(b)=\zl(a\cdot b)\quad\text{and}\quad
\zl(1_A)=\op{id}_M\,,$$
where $a,b\in A$ and where $1_A$ denotes the unit of $A$;
we often write $am$ instead of $\zl(a)(m)$, for $a\in A,\,m\in M$.
The condition of weight 0 for the map $\zl$ reads:
$$
\widetilde{am}= \widetilde{a}+ \widetilde{m}\,.
$$
As usual we inject $\K$ into $A$ by means of $K\ni
k\rightarrowtail k 1_A\in A$, so that $\zl(k)(m)=km$, $m\in M$.
 Graded
right modules over $A$ are defined similarly. Since $A$ is graded
commutative, any graded left $A$-module structure on $M$ defines a
right one,
$$ma:=(-1)^{\la\tilde m,\tilde a\ra}am\,,$$
and vice versa. Hence, we identify both concepts and speak just about graded
modules over a graded commutative algebra, as we do in the
commutative and supercommutative contexts.
Let us mention that one can in general realize a graded left module as a bimodule in many other ways, e.g. by setting
$$ ma := (-1)^{\la\tilde m +\zg,\tilde a\ra}am\,, $$
for any non-zero degree $\zg\in\zG$.
\medskip

Let now $M$ and $N$ be two graded $A$-modules.
Denote by $\sss{Hom}_A^{\zg}(M,N)$ the subspace of $\sss{Hom}^{\zg}_{\K}(V,W)$
consisting of $A$-linear maps $\ell:M\to N$ of weight $\zg$ that is
$$\ell(am)=
(-1)^{\la\zg,\tilde a\ra}a\,\ell(m)
\quad
\text{or, equivalently,}\quad
\ell(ma)=\ell(m)a$$
$$ \mbox{ and } \qquad \ell(M^{\zg'})\subset{}N^{\zg+\zg'}\,.$$
The space
$$\sss{Hom}_{A}(M,N)=\bigoplus_{\zg\in\zG}\sss{Hom}_A^{\zg}(M,N)$$
carries itself an obviously defined graded $A$-module structure.
The space
$$
\sss{End}_{A}(M):=\sss{Hom}_{A}(M,M)
$$
is a graded Lie algebra
with respect to the commutator (\ref{GCom}).
\medskip

Graded $A$-modules and $A$-linear maps of weight 0 form a category
{\tt Gr}$_\zG${\tt Mod}$_A$. Hence, the categorical Hom is the
vector space $\op{Hom}(M,N)=\sss{Hom}_A^0(M,N)$. \medskip

As $A$ is a
graded module over itself, the internal or inner (to {\tt Gr}$_\zG${\tt
Mod}$_A$) \sss{Hom} provides the notion of \textit{dual} module
$M^*=\sss{Hom}_A(M,A)$ of a graded $A$-module $M$. Let us also
mention that the categorical $\op{Hom}$ sets corresponding to
graded associative algebras (resp., graded associative unital
algebras, graded Lie algebras) are defined naturally as the sets
of those $\K$-linear maps of weight $0$ that respect the
multiplications (resp., multiplications and units,
brackets).\medskip

A {\it free graded $A$-module} is a graded $A$-module $M$ whose
terms $M^{\zg}$ admit a basis
$$B^{\zg}=(e_1^{\zg},\ldots,e_{r}^{\zg})\,.$$
Assume that the Abelian group $\zG$ is
of finite order $p$, and fix a basis $\{\zg_1,\ldots,\zg_{p}\}$.
Assume also that $M$ has a finite rank:
$\mathbf{r}=(r_{1},\ldots, r_{p})$, where
$r_{u}\in\mathbb{N}$ is the cardinality of
$B^{\zg_u}$. If $N$ is another free graded $A$-module of finite rank
$\mathbf{s}=(s_1,\ldots, s_{p})$ and basis
$({e'}_1^{\zg_k},\ldots,{e'}_{s_{k}}^{\zg_k})_k$, then every homogeneous
$A$-linear map $\ell\in\sss{Hom}_{A}(M,N)$ is represented by a
matrix $X$ defined by
$$
\sum_{k=1}^{p}\sum_{i=1}^{s_k}{e'}^{\zg_k}_i
(X_{ku})_{ij}:=\ell(e^{\zg_u}_j)\,,
$$
where $u\in\{1,\ldots,p\}$ and $j\in\{1,\ldots,r_u\}$.\medskip

Every homogeneous matrix can be written in the form:
\begin{equation}
\label{MatX}
X=\left(
\begin{array}{c|c|c}
X_{11} \Top &\; \ldots \; &X_{1p}\\[6pt]
\hline
\ldots \Top &\ldots&\ldots\\[3pt]
\hline
&&\\[-4pt]
X_{p1} &\ldots&X_{pp} \\[3pt]
\end{array}
\right)\,,
\end{equation}
where each $X_{ku}$ is a matrix of dimension
$s_k\times r_u$ with entries in $A^{-\zg_k+\zg_u+x}$.
We denote by $\sss{M}^{x}(\mathbf{s},\mathbf{r}; A)$
the set of \textit{homogeneous matrices} of degree $x\in\zG$
and
$$\sss{M}(\mathbf{s}, \mathbf{r}; A)=
\bigoplus_{x\in\G}\sss{M}^{x}(\mathbf{s}, \mathbf{r};A)\,.
$$

The set $\sss{M}^{x}(\mathbf{s},\mathbf{r}; A)$ is in
1-to-1 correspondence with the space
$\sss{Hom}_A^{x}(A^{\mathbf{r}},A^{\mathbf{s}})$ of all weight
$x$ $A$-linear maps between the free graded $A$-modules
$A^{\mathbf{r}}$ and $A^{\mathbf{s}}$ of rank $\mathbf{r}$ and
$\mathbf{s}$, respectively. This correspondence allows
transferring the vector space structure of the latter space to
weight~$x$ graded matrices. We thus obtain:
\begin{itemize}
\item
 the usual matrix sum of matrices;
 \item
  the usual multiplication of matrices.
  \end{itemize}
The \textit{multiplication of matrices by scalars} in $A$ is less obvious.
One has:
\begin{equation}
\label{AmoduleStruc}
a\,X=\left(
\begin{array}{c|c|c}
(-1)^{\la \tilde a,\zg_1\ra}a X_{11} \Top & \quad \ldots \;\;\;\; &(-1)^{\la \tilde a,\zg_1\ra}a X_{1p}\\[6pt]
\hline
\ldots \Top &\ldots&\ldots\\[3pt]
\hline
&&\\[-5pt]
(-1)^{\la \tilde a,\zg_p\ra}a X_{p1}&\ldots&(-1)^{\la \tilde a,\zg_p\ra}a X_{pp} \\[3pt]
\end{array}
\right)
\end{equation}
so that the sign depends on the row of a matrix.
Indeed, the graded $A$-module structure of $\sss{M}(\mathbf{s}, \mathbf{r}; A)$ is
induced by the $A$-module structure on $\sss{Hom}_A(A^{\mathbf{r}},A^{\mathbf{s}})$.
 \medskip

The space
$$
\sss{M}(\mathbf{r};A):=\sss{M}(\mathbf{r}, \mathbf{r}; A)\simeq
\sss{End}_A(A^{\mathbf{r}})
$$
is the most important example of the space of matrices.
This space is a graded $A$-module and a
graded associative unital algebra, hence a graded Lie algebra for
the graded commutator~(\ref{GCom}).
The even invertible matrices form a group that we denote by $\sss{GL}(\mathbf{r};A)$.\medskip

\subsection{$(\Z_2)^n$- and $(\Z_2)^{n+1}$-Gradings on Clifford Algebras}\label{CliS}
From now on, we will consider the following Abelian group:
$$\zG=(\Z_2)^n$$
of order $2^n$.
Elements of $(\Z_2)^n$ are identified with $n$-vectors with coordinates $0$ and $1$,
the element $0:=(0,\ldots,0)$ is the unit element of the group.
We will need the following two simple additional definitions related to $(\Z_2)^n$.

\begin{itemize}
\item
The group $(\Z_2)^n$ is equipped with the standard \textit{scalar product}
with values in $\Z_2 \,$:
\begin{equation}
\label{ScalPr}
\la\g,\g'\ra=\sum_{i=1}^n\g_i\g_i'\,.
\end{equation}
\item
An \textit{ordering} of the elements of $(\Z_2)^n$, such
that the first (resp., the last) $2^{n-1}$ elements are even
(resp., odd).
The order is termed \emph{standard} if in addition, the subsets of even and odd elements are ordered
lexicographically. For instance,
\beas
\Z_2 &=& \{ 0,1\}\,,\\
(\Z_2)^{2}&=& \{ (0,0),(1,1),(0,1),(1,0) \}\,,\\
(\Z_2)^{3}&=& \{ (0,0,0),(0,1,1),(1,0,1),(1,1,0)
,(0,0,1),(0,1,0),(1,0,0),(1,1,1)\}\,. \eeas

\end{itemize}
\medskip

The \textit{real Clifford algebra}
$\op{Cl}_{p,q}(\R)$ is the associative $\R$-algebra
generated by $e_i$, where $1\le i\le n$ and
$n=p+q$, of $\R^n,$ modulo the relations
$$
\begin{array}{rcl}
e_ie_j&=&-e_je_i\, ,
\quad i\neq j\;,\\ [4pt]
e_i^2&=&
\begin{cases}+1\,,\quad i\le p\\ -1\,,\quad
i>p\,.\end{cases}
\end{array}
$$
The pair of integers $(p,q)$ is called the \emph{signature}.
Note that, as a vector space, $\op{Cl}_{p,q}(\R)$ is isomorphic to the
Grassmann algebra $\bigwedge\la e_1,\ldots,e_n\ra$ on the chosen
generators. Furthermore,  $\op{Cl}_{p,q}(\R)$ is often understood as
quantization of the Grassmann algebra (in the same sense as the Weyl algebra is a quantization of the symmetric algebra).\medskip

Real Clifford algebras can be seen as graded commutative algebras essentially in two different ways.

A $(\Z_2)^n$-grading on $\op{Cl}_{p,q}(\R)$ was defined in
\cite{AM1} by setting for the generators
$$
\widehat{e_i}=\left(0,\ldots,0,1,0,\ldots,0\right)\,,
$$
where $1$ occupies the $i$-th position.
However, the graded commutativity condition (\ref{GCon}) is not
satisfied with respect to the standard scalar product
(\ref{ScalPr}), which has to be replaced by another binary
function on $(\Z_2)^n$, see \cite{AM1}.

A $(\Z_2)^{n+1}_{\;0}$-grading on $\op{Cl}_{p,q}(\R)$ has been
considered in \cite{OM}. This grading coincides with the preceding
Albuquerque-Majid degree, if one identifies $(\Z_2)^n$ with the
even subgroup $(\Z_2)^{n+1}_{\;0}$ of $(\Z_2)^{n+1}$. Indeed, the
new degree is defined by
\begin{equation}
\label{GoodGrad}
\widetilde{e_i}:=\left(0,\ldots,0,1,0,\ldots,0,1\right)\,.
\end{equation}
An advantage of this ``even'' grading is that the condition
(\ref{GCon}) is now satisfied with respect to the standard scalar
product. It was proven in \cite{OM} that the defined
$(\Z_2)^{n+1}_{\;0}$-grading on $\op{Cl}_{p,q}(\R)$ is universal
in the following sense: {\it every simple finite-dimensional
associative graded-commutative algebra is isomorphic to a Clifford
algebra equipped with the above $(\Z_2)^{n+1}_{\;0}$-grading}.

\begin{ex}\label{quatgrad}
The $(\Z_2)^3_{\;0}$-grading of the quaternions
$\mathbb{H}=1\R\oplus \qi\R\oplus \qj\R\oplus \qk\R$ is defined by:
\begin{equation}
\label{DegH}
\begin{array}{rcl}
\degr{1}&=&(0,0,0)\,,\\[4pt]
\degr{\qi}&=&(0,1,1)\,,\\[4pt]
\degr{\qj}&=&(1,0,1)\,,\\[4pt]
\degr{\qk}&=&(1,1,0)\,,
\end{array}
\end{equation}
see \cite{MO} for more details.
\end{ex}

\begin{rem}
It is natural to understand Clifford algebras as \textit{even algebras}.
Moreover, sometimes it is useful to consider larger graded algebras
that contain a given Clifford algebra as an even part, see \cite{OM11}.
It is therefore natural to use the even grading (\ref{GoodGrad}).
\end{rem}

\section{$(\Z_2)^n$-Graded Trace}\label{ProofS}

In this section we introduce
the notion of \textit{graded trace} of a matrix over a
$(\Z_2)^n$-graded commutative algebra $A$
that extends the notion of supertrace.
Although the proof of the main result is quite elementary,
this is the first important ingredient of our theory.
Let us also mention that the notion of trace is missing
in the literature on quaternionic matrices
(as well as on matrices with coefficients in Clifford algebras).

\subsection{Fundamental Theorem and Explicit Formula}\label{TrThm}

The first main result of this paper is as follows.

\begin{thm}
\label{TThmGen}
There exists a unique (up to multiplication by a scalar of weight $0$)
$A$-linear graded Lie algebra
homomorphism
$$
\op{\zG tr}:\sss{M}(\mathbf{r}; A)\to A\,,
$$
defined for a homogeneous matrix $X$ of degree $x$ by
\begin{equation}
\label{TheTrace}
\op{\zG tr}(X)=\sum_k(-1)^{\la\zg_k+x,\,\zg_k\ra}\op{tr}(X_{kk})\,,
\end{equation}
where $\op{tr}$ is the usual trace and where $X_{kk}$ is a graded block of $X$,
see formula (\ref{MatX}).
\end{thm}

Let us stress the fact that the term ``homomorphism'' must be understood in the categorical sense and means homomorphism of weight 0.
For any $X\in \sss{M}(\mathbf{r}; A)$, we refer to $\op{\zG tr}(X)\in A$
as the graded trace of $X$.
Of course, if $A$ is a usual supercommutative ($\Z_2$-graded) algebra,
we recover the classical supertrace $\op{str}$.

\begin{proof}
It is straightforward to check that formula (\ref{TheTrace}) is
$A$-linear and indeed defines a graded Lie algebra morphism. Let
us prove uniqueness. \medskip

Recall that a homogeneous matrix $X\in \sss{M}^x(\mathbf{r}; A)$
is a matrix that contains $p\times p$ blocks $X_{ku}$ of dimension
$r_k\times r_u$ with entries in $A^{\zg_k+\zg_u+x}$. We denoted
the entry $(i,j)$ of block $X_{ku}$, located on block row $k$ and
block column $u$, by $(X_{ku})_{ij}$. Let us emphasize that if we
view $X$ as an ordinary $\sum_k r_k\,\times \sum_k r_k$ matrix, we
denote its entries by $x_{\za\zb}$. \medskip

Let $E_{\za\zb}\in \sss{M}(\mathbf{r}; A)$ be the matrix containing $1_A$
in entry $(\za,\zb)$ and zero elsewhere. As any row index $\za$ determines
 a unique block row index $k$ and therefore a unique weight $w_{\za}:=\zg_k$,
 matrix $E_{\za\zb}$ is homogeneous of weight $w_{\za}+w_{\zb}$.
 It is easily seen that $E_{\za\zb}E_{\zh\ze}$ equals $E_{\za\ze}$, if $\zb=\zh$,
 and vanishes otherwise.\medskip

In view of Equation (\ref{AmoduleStruc}), any graded matrix
$X\in\sss{M}^x(\mathbf{r}; A)$ reads
$$X=\sum_{\za,\zb}(-1)^{\la
w_{\za}+w_{\zb}+x,w_{\za}\ra}x_{\za\zb}E_{\za\zb}\,.$$ It
follows from the graded $A$-module morphism property of the graded
trace that this functional is completely determined by its values
on the matrices $E_{\za\zb}$. Moreover, the graded Lie algebra
property entails
$$
\op{\zG tr}(E_{\za\zb})= \op{\zG
tr}(E_{\za 1}E_{1\zb})=(-1)^{\la
w_{\za}+w_1,w_1+w_{\zb}\ra}\op{\zG tr}(E_{1\zb}E_{\za
1})=(-1)^{\la w_{\za},w_{\zb}\ra}\zd_{\za\zb}\op{\zG
tr}(E_{11})\,,
$$
where $\zd_{\za\zb}$ is Kronecker's symbol. When
combining the two last results, we get
$$
\op{\zG tr}(X)=\op{\zG
tr}(E_{11})\sum_{\za}(-1)^{\la
w_{\za}+x,w_{\za}\ra}x_{\za\za}=
\sum_k(-1)^{\la
\zg_k+x,\zg_k\ra}\op{tr}(X_{kk})\,.
$$
Hence the uniqueness.
\end{proof}

\begin{rem}
Thanks to its linearity, the graded trace (\ref{TheTrace}) is well-defined for an \textit{arbitrary}
(not necessarily homogeneous or even) matrix $X$.
This will not be the case for the graded determinant or graded Berezinian.
In this sense, the notion of trace is more universal.
On the other hand, in algebra conditions of invariance can always be formulated infinitesimally
so that the trace often suffices.
\end{rem}

In Section \ref{ExSect}, we will give a number of examples of traces
of quaternionic matrices.

\subsection{Application: Lax Pairs}

Let us give here just one application of Theorem \ref{TThmGen}.

\begin{cor}
\label{Lax}
Given two families of even matrices $X(t),Y(t)\in \sss{M}(\mathbf{r}; A)$
(smooth or analytic, etc.) in one real or complex parameter $t$ satisfying the equation
$$
\frac{d}{dt} X=\left[X,Y\right]\,,
$$
the functions $\op{\zG tr}(X),\,\op{\zG tr}(X^2),\ldots$ are independent of $t$.
\end{cor}

\begin{proof}
The function $\op{\zG tr}$ obviously commutes with $\frac{d}{dt}$,
therefore
$$
\frac{d}{dt}\op{\zG tr}(X)=\op{\zG tr}\left(\left[X,Y\right]\right)=0\,.
$$
Furthermore, let us show that
$$
\frac{d}{dt}\op{\zG tr}(X^2)=\op{\zG tr}\left(\left[X,Y\right]X+X\left[X,Y\right]\right)=0\,.
$$
Indeed, by definition of the commutator (\ref{GCom}) one has:
$$
\left[X,Y\right]X=XYX-(-1)^{\la{}y,x\ra}YXX=\left[X,\,YX\right]
$$
and
$$ X \left [X,Y\right]=XXY-(-1)^{\la y,x\ra}XYX=\left[X,\,XY \right] $$
due to the assumption that $X$ is even, i.e., $\la x,x\ra=0$.

This argument can be generalized to prove that $\frac{d}{dt}
\op{\zG tr}(X^k)=0$ for higher $k$, since we find by induction
that
 $$ \frac{d}{dt} X^k = \sum_{i=1}^k \lpq X, X^{i-1}YX^{k-i} \rpq \,.$$

\end{proof}

The above statement is an analog of the Lax representation
that plays a crucial role in the theory of integrable systems.
The functions $\op{\zG tr}(X),\,\op{\zG tr}(X^2),\ldots$
are first integrals of the dynamical system $\frac{d}{dt}X=[X,Y]$ that often suffice to prove its integrability.
Note that integrability in the quaternionic and more generally
Clifford case is not yet understood completely.
We hope that interesting examples of integrable systems
can be found within the framework of graded algebra.

\section{$(\Z _2)^n$-Graded Determinant of Purely Even Matrices of Degree 0}\label{BerSect}

Let $A$ be a \textit{purely even} $(\Z _2)^n$-graded commutative
i.e., a $(\Z_2)^{n}_{\,0}$-graded commutative algebra. We also
refer to matrices over $A$ as \emph{purely even}
$(\Z_2)^{n}$-graded or $(\Z_2)^{n}_{\;0}$-graded matrices. Their
space will be denoted by $\textsf{M}(\mathbf{r}_0;A)$, where
$\mathbf{r}_0 \in \N^q$, $q=2^{n-1}$.

\subsection{Statement of the Fundamental Theorem}

As in usual Superalgebra, the case of purely even matrices is
special, in the sense that we obtain a concept of determinant
which is polynomial (unlike the general Berezinian).

\begin{thm}\label{thmgdet}
(i)
There exists a unique map $\gdet:\sss{M}^0(\mathbf{r}_0;A)\to
A^{0}$ that satisfies:
\begin{enumerate}
    \item For all $X,Y\in \sss{M}^0(\mathbf{r}_0;A)$, we have $$ \gdet(XY)=\gdet(X)\cdot\gdet(Y)\,. $$
    \item If $X$ is block-diagonal, then
    $$ \gdet(X)=\prod_{k=1}^{q} \det(X_{kk})\,.$$
    \item If $X$ is block-unitriangular, then $\gdet(X)=1\,.$
\end{enumerate}
(ii)
For any matrix $X\in\sss{M}^0(\mathbf{r}_0;A)$, the value $\gdet(X)$ is linear in the rows and columns of $X$ and therefore it is polynomial.
\end{thm}

We refer to $\gdet(X)\in A^{0}$, where $X\in\sss{M}^0(\mathbf{r}_0;A)$, as the {\it graded determinant} of $X$.
\medskip

To prove the theorem, we first work formally (see \cite{Olv06}),
i.e., we assume existence of inverse matrices of all square
matrices. Hence, in one sense, we begin by working on an open
dense subset of $\sss{M}^0(\mathbf{r}_0;A)$. We show that the
graded determinant $\gdet$ formally exists and is unique, then we
use this result to give evidence of the fact that $\gdet$ is
polynomial. The latter polynomial will be our final definition of
$\gdet$ and Theorem \ref{thmgdet} will hold true not only formally
but in whole generality.

\subsection{Preliminaries}
To find a (formal) explicit expression of the graded determinant, we use an UDL decomposition (i.e., a factorization into an upper unitriangular matrix $U$, a diagonal matrix $D$, and a lower unitriangular matrix $L$). For matrices over a not necessarily commutative ring the entries of the UDL factors are tightly related to quasideteminants.
To ensure independent readability of the present text, we recall the concept of quasideterminants (see \cite{GR91} and \cite{GGRW05} for a more detailed and extensive survey on the subject).

\subsubsection{Quasideterminants}
Quasideterminants are an important tool in Noncommutative Algebra;
known determinants with noncommutative entries are products of
quasiminors.
For matrices over a non-commutative ring the entries of
the UDL factors are tightly related to quasideterminants.
 In general, a quasideterminant is a rational function
in its entries; in the commutative situation, a quasideterminant
is not a determinant but a quotient of two determinants.\medskip

Let $R$ be a unital (not necessarily commutative) ring and let
$X\in\op{gl}(r;R),$ $r\in\mathbb{N}\setminus\{0\}$. Denote by
$X^{i,j}$, $1 \leq i,j \leq r$, the matrix obtained from $X$ by
deletion of row $i$ and column $j$. Moreover, let $r_{i}^{j}$
(resp., $c_{j}^{i}$) be the row $i$ (resp., the column $j$)
without entry $x_{ij}$.
\begin{defi}\label{defQuasidet}
If $X\in\op{gl}(r; R)$ and if the submatrix $X^{i,j}$ is
invertible over $R$, the {\it quasideterminant} $(i,j)$ of $X$ is
the element $\left| X \right|_{ij}$ of $R$ defined by \be \left| X
\right|_{ij}:= x_{ij} -  r_{i}^{j} (X^{i,j})^{-1} c_{j}^{i}\,.
\ee
\end{defi}

Any partition $r=r_{1}+r_{1}+ \cdots
+r_{p}$ determines a $p\times p$ block decomposition
$X=(X_{ku})_{1\leq k,u\leq p}$ with square diagonal blocks.
According to common practice in the literature on
quasideterminants, the entries $(X_{ku})_{ij}$ of the block
matrices $X_{ku}$ are in the following numbered consecutively,
i.e., for any fixed $k,u$,
$$
1+\sum_{l < k} r_{l} \leq i \leq \sum_{l \leq k}
r_{l}\quad\text{and}\quad 1+\sum_{l < u} r_{l} \leq j \leq \sum_{l
\leq u} r_{l}\,.
$$

The most striking property of quasideterminants is the
\emph{heredity principle}, i.e., ``a quasideterminant of a quasideterminant is a
quasideterminant''.
The following statement is proved in \cite{GGRW05}.\medskip

\noindent
{\bf Heredity principle.}
{\it
Consider a decomposition of $X\in\op{gl}(r; R)$ with square
diagonal blocks, a fixed block index $k$, and two fixed indices $i,j$ such that
$$
1+\sum_{l < k} r_{l} \leq i\,,\,
j \leq \sum_{l \leq k} r_{l}\,.
$$
 If the quasideterminant $\left| X \right|_{kk}$ is defined, then the
quasideterminant $||X | _{kk}|_{ij}$ exists if and only if the
quasideterminant $\left| X \right|_{ij}$ does. Moreover, in this
case}
\be |\left| X \right|_{kk}|_{ij} = \left| X \right|_{ij}\,.
\label{HP1}\ee
\medskip

Observe that $|X|_{kk}$ is a quasideterminant, not over a unital
ring, but over blocks of varying dimensions. Such blocks can only
be multiplied, if they have the appropriate dimensions, the
inverse can exist only for square blocks. However, under the usual
invertibility condition, all operations involved in this
quasideterminant make actually sense and
$|X|_{kk}\in\op{gl}(r_k;R)$.\medskip

In the following, we will need the next corollary of the
heredity principle.

\begin{cor}
\label{HPAugmented}
Consider a decomposition of $X\in\op{gl}(r; R)$
with square diagonal blocks, a block index $k$, and two indices $i,a$ (resp., $j,b$)
in the range of block $X_{kk}$, such that $i\neq a$ (resp., $j
\neq b$). If $|X|_{kk}$ exists, the {\small LHS} of Equation
(\ref{HP+}) exists if and only if the {\small RHS} does, and in
this case we have \be \label{HP+} | (| X |_{kk})^{i,j} |_{ab} =
\left| X^{i,j} \right|_{ab}\,. \ee
\end{cor}
    \begin{proof}
    Assume for simplicity that $k=1$ and set $$X=\left(%
\begin{array}{cc}
  X_{11} & B \\[4pt]
  C & D \\
\end{array}%
\right)\,.$$ By definition,
    $$ \left| X \right|_{11}= X_{11}-BD^{-1}C\,.$$
    Clearly, $$( \left| X \right|_{11})^{i,j}= (X_{11})^{i,j}-B^{i,0}D^{-1}C^{0
    ,j}\,,$$ where $0$ means that no column
    or row has been deleted, coincides with $\left| X^{i,j} \right|_{11}$.
    The claim now follows from the heredity principle.
    \end{proof}

Heredity shows that quasideterminants handle matrices
over blocks (square diagonal blocks assumed) just the same as
matrices over a ring. Quasideterminants over blocks (square
diagonal blocks) that have themselves entries in a field were
studied in \cite{Olv06}. In view of the preceding remark, it is
not surprising that the latter theory coincides with that of
\cite{GGRW05}. Especially the heredity principle holds true for
decompositions of block matrices. Moreover, the nature of the
block entries is irrelevant, so that the results are valid for
block entries in a ring as well.\medskip

\begin{ex} Let
$$X=\lp \begin{array}{ccc}
x & a & b \\
c & y & d \\
e & f & z \\
 \end{array}\rp\in\op{gl(3,R)}\,,
 $$
 where $R$ is as above a unital ring. In this example, we work formally, i.e., without addressing the question of existence
of inverses. A short computation allows to see that the formal
inverse of a $2\times 2$ matrix over $R$ is given by
\be
\label{InvForm}
\lp\begin{array}{cc} y & d \\[6pt]
f & z
\end{array}\rp^{-1}=\lp
\begin{array}{rl} (y-dz^{-1}f)^{-1} &
-(y-dz^{-1}f)^{-1}dz^{-1} \\[8pt]
 -z^{-1}f(y-dz^{-1}f)^{-1} & z^{-1}+
z^{-1}f(y-dz^{-1}f)^{-1}dz^{-1}
\end{array}\rp\,.
\ee It then
follows from the definition of quasideterminants that
\be \left|X
\right|_{11} = x- b z^{-1}e
-(a-bz^{-1}f)(y-dz^{-1}f)^{-1}(c-dz^{-1}e) \label{herex}\,.\ee
Note that, when viewing matrix $X$ as block matrix $$X=\lp
\begin{array}{cc|c}
x & a & b \\
c & y & d \\
\hline
e & f & z \\
 \end{array}\rp\,,$$ with square diagonal blocks, the quasideterminant
$$
\left| \left| X \right|_{11} \right|_{11}=
\left|\begin{array}{cc} x-bz^{-1}e & a-bz^{-1}f \\
c-dz^{-1}e & y-dz^{-1}f \end{array} \right|_{11}\,,
$$
where the interior subscripts refer to the $2\times 2$ block decomposition
can be obtained without Inversion Formula (\ref{InvForm})
and coincides with the quasideterminant (\ref{herex}),
as claimed by the heredity principle.\end{ex}

\begin{rem} This example corroborates the already mentioned fact
that a quasideterminant with respect to the ordinary row-column
decomposition (resp., to a block decomposition) is a
rational expression (resp., a block of rational expressions) in
the matrix entries. \end{rem}

\subsubsection{UDL Decomposition of Block Matrices with Noncommutative Entries}\label{UDLsec}
 An {\it UDL decomposition} of a square matrix is a factorization into an upper unitriangular
 (i.e., triangular and all the entries of the diagonal are equal to $1$)
 matrix $U$, a diagonal matrix $D$, and a lower unitriangular matrix $L$.
 In this section we study existence and uniqueness of a block UDL decomposition
 for invertible block matrices with square diagonal blocks that have entries
in a not necessarily commutative ring $R.$

\begin{defi} An invertible block matrix $X=(X_{ku})_{k,u}$
with square diagonal blocks $X_{kk}$ and entries $x_{\za\zb}$ in $R$ is called {\it regular}
if and only if it admits a block UDL decomposition.\end{defi}

\begin{lem} If $X$ is regular, its UDL decomposition is unique.\end{lem}

\begin{proof}
If $UDL=X=U'D'L'$ are two such decompositions, then
$U'^{-1}UD=D'L'L^{-1}$ and $\mathcal{U}=U'^{-1}U$ (resp., $\mathcal{L}=L'L^{-1}$)
is an upper (resp., lower) unitriangular matrix.
Since $\mathcal{U}D$ (resp., $D'\mathcal{L}$) is an upper (resp., lower) triangular matrix
with diagonal $D$ (resp., $D'$), we have $D=D'$.
The invertibility of $X$ entails that $D$ is invertible,
so $\mathcal{U}=D\mathcal{L}D^{-1}$,
where the {\small LHS} (resp., {\small RHS}) is upper (resp., lower) unitriangular.
It follows that $U=U'$ and $L=L'$.\end{proof}

\begin{prop}\label{UDL}
An invertible $p\times p$ block matrix $X=(X_{ku})_{k,u}$ with square diagonal blocks $X_{kk}$ and entries $x_{\za\zb}$ in $R$ is {\it regular} if and only if its principal block submatrices $$X^{1,1}, X^{12,12}, \ldots, X^{12\ldots (p-1),12\ldots (p-1)}=X_{pp}$$ are all invertible over $R$. In this case, $X$ factors as $X={\frak U D^{-1} \frak L}$, where
$$
\frak U=\left(\begin{array}{ccccc}|X|_{11}& |X^{2,1}|_{12}&
|X^{23,12}|_{13}&\ldots & X_{1p}\\[4pt] &
|X^{1,1}|_{22}&|X^{13,12}|_{23}&\ldots & X_{2p}\\[4pt]
&& |X^{12,12}|_{33}&\ldots &X_{3p}\\&&&\ddots
&\vdots\\&&&&X_{pp}
\end{array} \right)\,,
$$
$$
D=\left(\begin{array}{ccccc}|X|_{11}&&&&\\ &
|X^{1,1}|_{22}&&&\\
&& |X^{12,12}|_{33}& &\\&&&\ddots &\\&&&&X_{pp}\end{array}
\right)\,,
$$
and
$$
\frak
L=\left(\begin{array}{lllll}|X|_{11}&&&&\\ [4pt]
|X^{1,2}|_{21}&
|X^{1,1}|_{22}&&&\\[4pt]
|X^{12,23}|_{31} &|X^{12,13}|_{32}&
|X^{12,12}|_{33}&&\\
\vdots&\vdots&\vdots&\ddots
&\\X_{p1}&X_{p2}&X_{p3}&\ldots &X_{pp}\end{array} \right)\,.
$$
Observe that $\frak U_{ku}=|X^{1\ldots\hat k\ldots
u,1\ldots (u-1)}|_{ku}$ and $ \frak L_{uk}=|X^{1\ldots(u-1),1\ldots \hat k\ldots u}|_{uk} $ (for $k\leq u$). Matrix $X$ factors also as $$X=UDL\,,$$
where $$U=\frak U D^{-1}\quad (\text{resp.,}\;\; L=D^{-1}\frak L)$$ is an upper (resp., lower) unitriangular matrix.
\end{prop}

\begin{proof} The cases $p=1$ and $p=2$ are straightforward. Indeed, if $X$ is an invertible $2\times 2$ block matrix, if we write for simplicity $A$ (resp., $B,C,D$) instead of $X_{11}$ (resp., $X_{12}, X_{21}, X_{22}$), denote identity blocks by $\mathbb{I}$, and if submatrix $D$ is invertible, we have
\beas \label{UDL2} \lp \begin{array}{cc} A & B \\[4pt] C & D
\end{array} \rp & = & \lp \begin{array}{cc} \mathbb{I} & BD^{-1}
\\[4pt] 0 & \mathbb{I}
\end{array} \rp \lp \begin{array}{cc} A-BD^{-1}C & 0 \\[4pt] 0 & D
\end{array} \rp
\lp \begin{array}{cc} \mathbb{I} & 0 \\[4pt] D^{-1}C & \mathbb{I}\end{array} \rp\\[4pt]
& = & \lp \begin{array}{cc}  A-BD^{-1}C & B \\[4pt] 0 & D \end{array}
\rp \lp \begin{array}{cc}  A-BD^{-1}C & 0 \\[4pt] 0 & D \end{array} \rp
^{-1} \lp \begin{array}{cc}  A-BD^{-1}C  & 0 \\[4pt] C & D \end{array}
\rp\eeas (note that the first equality is valid even if $X$ is not necessarily invertible). Conversely, if $X$ is regular, its UDL
decomposition is necessarily the preceding one, so that $D$ is
actually invertible.\medskip

For $p=3$, i.e., for an invertible $3\times 3$ block matrix $X$
such that $X^{1,1}, X^{12,12}=X_{33}$ are invertible, the UDL part of Proposition \ref{UDL} states that matrix $X$ is given by
\beas\small{ \lp \begin{array}{ccc} \mathbb{I} &
|X^{2,1}|_{12}|X^{1,1}|_{22}^{-1} & X_{13}X_{33}^{-1} \\[4pt]
&\mathbb{I} & X_{23}X_{33}^{-1}\\[4pt]
&&\mathbb{I} \end{array}
\rp
\lp
\begin{array}{ccc}|X|_{11} && \\[4pt]
 & |X^{1,1}|_{22} & \\[4pt]
 && X_{33}\end{array}
\rp\lp
\begin{array}{ccc}
\mathbb{I} & &\\[4pt]
|X^{1,1}|_{22}^{-1}|X^{1,2}|_{21}& \mathbb{I} &\\[4pt]
X_{33}^{-1}X_{31} & X_{33}^{-1}X_{32} & \mathbb{I}
\end{array}\rp}\,.\eeas

\vspace{5mm}\noindent Observe first that the proven UDL
decomposition for $p=2$, applied to $X^{1,1}$, entails that
$|X^{1,1}|_{22}$ is invertible. It is now easily checked that all
the expressions involved in the preceding $3\times 3$ matrix
multiplication make sense (in particular the quasideterminants of
the rectangular matrices $X^{2,1}$ and $X^{1,2}$ are
well-defined). The fact that this product actually equals $X$ is
proved by means of a
$$ \lp
\begin{array}{cc} 2\times 2 & 2\times1\\[4pt]
 1\times 2 & 1\times 1 \end{array} \rp$$
 redivision. The result
is then a consequence of two successive applications of the
aforementioned $2\times 2$ UDL decomposition and of the heredity
principle.\medskip

\noindent The ${\frak U}D^{-1}{\frak L}$ decomposition of $X$ follows immediately from its just proven UDL decomposition (again, invertibility of $X$ is not needed for the proof of the UDL decomposition). Conversely, if matrix $X$ is regular, we see that $|X^{1,1}|_{22}$ and $X_{33}$ are invertible, then, from the case $p=2$, that $X^{1,1}$ is invertible. The passage from $p>2$ to $p+1$ is similar to the passage from $p=2$ to $p+1=3.$
\end{proof}

\subsection{Explicit Formula in Terms of Quasideterminants}

Recall that a degree 0 $(\Z _2)^n_{\;0}$-graded matrix $X \in
\textsf{M}^0(\mathbf{r}_0;A)$ is a $q \times q$ block matrix,
where $q:=2^{n-1}$ is the order of $(\Z _2)^n _{\; 0}$. The
entries of a block $X_{ku}$ of $X$ are elements of $A^{\zg_k +
\zg_u}$. Every such matrix $X$ admits (formally) an UDL
decomposition with respect to its block structure. \medskip

It follows from Proposition \ref{UDL} that, if the graded determinant of any
$X\in\textsf{M}^0(\mathbf{r}_0;A)$ exists, then it is
equal to
\be \label{gdetExpl} \gdet(X)= \prod_{k=0}^{q-1}\det |
X^{1\ldots k \, , \, 1\ldots k }|_{k+1\;k+1}\,. \ee
Hence, the graded determinant is (formally) unique. Observe that for
$n=1$ and $n=2$ it coincides with the classical determinant (for
$n=2$, see UDL decomposition). This observation is natural, since the
entries of the considered matrices are in these cases elements of
a commutative subalgebra of $A$. Note also that the graded
determinant defined by Equation (\ref{gdetExpl}) satisfies
Conditions (2) and (3) of Theorem \ref{thmgdet}. To prove (formal)
existence, it thus suffices to check Condition (1) (for
$n>2$).\medskip

The proof of multiplicativity of $\gdet$
given by (\ref{gdetExpl}) is based on an induction on $n$ that
relies on an equivalent inductive expression of $\gdet$.

This expression is best and completely understood, if detailed for
low $n$, e.g., $n=3$. As recalled above, a purely even
$(\Z_2)^{3}$-graded matrix $X$ of degree $0$ is a $4 \times 4$
block matrix. The degrees of the blocks of $X$ are
\begin{equation} \left(
 \begin{array}{cc|cc}
 {000}\Top &{011}&{101}&{110}\\[6pt]
 {011}&000&{110}&{101}\\[6pt]\hline
 {101}\Top &{110}&{000}&{011}\\[6pt]
 {110}&{101}&{011}&{000}\\[6pt]
 \end{array}
 \right)\,,
\end{equation}

\noindent so that, if we consider the suggested $2\times 2$
redecomposition \be\label{FirstDecomposition} \frak{X}=\left(
 \begin{array}{c|c}
 \frak{X}_{11}\Top & \frak{X}_{12}\\[6pt]
 \hline
 \frak{X}_{21}\Top & \frak{X}_{22}\\[6pt]
 \end{array}
 \right)\ee

\noindent of $X$, the quasideterminant $\left| \frak{X}
\right|_{11}$ and the block $\frak{X}_{22}$ (as well as products
of the type $\frak{X}_{12}\frak{X}_{22}^{-1}\frak{X}_{21}$,
$\frak{X}_{12}\frak{X}_{21},\ldots$ -- this observation will be
used below) can be viewed as purely even $(\Z_2)^{n-1}$-graded
matrices of degree $0$. It follows that the inductive expression
\be \gdet(X)= \gdet(\left| \frak{X} \right|_{11})\cdot
\gdet(\frak{X}_{22}) \label{recgdet}\ee actually makes sense. To
check its validity, observe that the {\small RHS} of
(\ref{recgdet}) reads, for $n=3$,
$$\det||\frak{X}|_{11}|_{11}\cdot\det|(|\frak{X}|_{11})^{1,1}|_{22}\cdot\det|\frak{X}_{22}|_{11}\cdot\det|(\frak{X}_{22})^{1,1}|_{22}\,,$$
where the indices in $|\frak{X}|_{11}$ and $\frak{X}_{22}$ (resp.,
the other indices) correspond to the $2\times 2$ redecomposition
$\frak X$ of $X$ (resp., the $2\times 2$ decomposition of
$|\frak{X}|_{11}$ and $\frak{X}_{22}$). When writing this result
using the indices of the $4\times 4$ decomposition of $X$, as well
as the Heredity Principle, see Equations (\ref{HP1}) and
(\ref{HP+}), we get
$$\det|X|_{11}\cdot\det|X^{1,1}|_{22}\cdot\det|X^{12,12}|_{33}\cdot\det|X^{123,123}|_{44}
=\gdet(X)\,.$$

Let us mention that the {\small LHS} and {\small RHS} $\gdet$-s in
Equation (\ref{recgdet}) are slightly different. The determinant
in the {\small LHS} is the graded determinant of a purely even
$(\Z_2)^{n}$-graded matrix, whereas those in the {\small RHS} are
graded determinants of purely even $(\Z_2)^{n-1}$-graded matrices.

\vspace{4mm} To prove multiplicativity of the graded determinant
$\gdet$ defined by (\ref{gdetExpl}) and (\ref{recgdet}),
we need the next lemma.

\begin{lem} \label{AB BA}
Let $X$ and $Y$ be two $(\Z_2)^{n}_{\;0}$-graded matrices of degree $0$ of the same
dimension. If $\frak{X}_{12}$ or $\frak{Y}_{21}$ (see Equation (\ref{FirstDecomposition})) is elementary,
i.e., denotes a matrix that contains a unique nonzero element, then
\be \label{presyl} \gdet\lp
\I+\frak{X}_{12}\frak{Y}_{21}\rp= \gdet \lp \I
+\frak{Y}_{21}\frak{X}_{12}\rp\,. \ee
\end{lem}

    \begin{proof}
    In view of the above remarks, it is clear that both graded determinants are (formally) defined.
    Assume now that $\frak{X}_{12}$ is
    elementary and has dimension $R\times S$, use numerations
    from 1 to $R$ and 1 to $S$, and denote the position of the unique nonzero element
    $x$ by $(r,s)$.

   One has:
    $$
\I+\frak{X}_{12}\frak{Y}_{21}= \left(
 \begin{array}{ccccc}
 1&&&&\\[6pt]
 &\ddots&&&\\[6pt]
 x{Y}_{s1}&\hdots&\boxed{\begin{array}{ccccc}1&&&&\\&\ddots &&&\\\hdots&& 1+x Y_{sr}&&\hdots\\&&&\ddots&\\&&&& 1\end{array}}&\hdots &xY_{sR}\\[6pt]
 &&&\ddots&\\[6pt]
 &&&& 1
 \end{array}
 \right)\,,$$\normalsize

\noindent where the element $1+xY_{sr}$ is located at position
$(r,r)$, and Equation (\ref{gdetExpl}) entails that
$$\gdet\lp
\I+\frak{X}_{12}\frak{Y}_{21}\rp=1+xY_{sr}\,.$$ A similar
computation shows that $$\gdet\lp
\I+\frak{Y}_{21}\frak{X}_{12}\rp=1+Y_{sr}x\,.$$

\noindent Since the elements $x$ at position $(r,s)$ in
$\frak{X}_{12}$ and $Y_{sr}$ at position $(s,r)$ in
$\frak{Y}_{21}$ have the same even degree, they commute.
\end{proof}

We are now prepared to prove multiplicativity of $\gdet$. Let us
begin by stressing that we will have to consider $2^n\times 2^n$
UDL decompositions $X_UX_DX_L$, as well as $2\times 2$ UDL
decompositions ${\frak X}_U{\frak X}_D{\frak X}_L$. Whereas an
upper unitriangular matrix ${\frak X}_U$ is also an upper
unitriangular matrix $X_U$ (and similarly for lower unitriangular
matrices), a diagonal matrix $X_D$ is also of the type ${\frak
X}_D$ (but the converses are not valid). Such details must of
course be carefully checked in the following proof, but, to
increase its readability, we refrain from explicitly mentioning
them.\medskip

Assume now that multiplicativity holds true up to $n$ ($n\ge 2$)
and consider the case $n+1$. If $X,Y$ denote
$(\Z_2)^{n+1}_0$-graded matrices, we need to show that
$\gdet(XY)=\gdet(X)\cdot\gdet(Y)$.
\medskip

(i) Let first $Y$ be lower unitriangular and set $X=X_UX_DX_L.$
Since $X_LY$ is again lower unitriangular, we have
$$\gdet(XY)=\gdet(X_D)=\gdet(X)\cdot\gdet(Y)\;.$$\vspace{1mm}

(ii) Assume now that $Y$ is diagonal,
$$
Y=\left(\begin{array}{cc}\frak{Y}_{1}&\\&\frak{Y}_{2}\end{array}\right)\;,
$$
where $\frak{Y}_1$ and $\frak{Y}_2$ are $(\Z_2)^{n}_0$-graded, and
let $X={\frak X}_U{\frak X}_D{\frak X}_L,$
$$
{\frak X}_D=\left(
\begin{array}{cc} \frak{X}_{1} &  \\[4pt]  & \frak{X}_{2} \end{array}\right),
\qquad
{\frak X}_{L}=\lp \begin{array}{cc} \I &  \\[4pt]
\frak{X}_{3} & \I
\end{array} \rp.  $$
Then, $ XY={\frak X}_{U}\,\lp{\frak X}_{D}\,{\frak Y}\rp\;{\frak
Z}\;, $ with \beas
{\frak Z}=\lp \begin{array}{cc} \I & \\[4pt] \lp
\frak{Y}_{2}\rp^{-1}\frak{X}_{3} \, \frak{Y}_{1} & \I
\end{array} \rp\;. \eeas
Since ${\frak X}_D{\frak Y}$ is block-diagonal, we get
$$\gdet(XY)=\gdet\lp \begin{array}{cc}
\frak{X}_{1}\frak{Y}_1 &  \\[4pt]  & \frak{X}_{2}\frak{Y}_2 \end{array}
\rp=\gdet(\frak{X_1Y_1})\cdot\gdet(\frak{X_2Y_2})$$
$$=\gdet({\frak X}_D)\cdot\gdet(Y)=\gdet(X)\cdot\gdet(Y)\;,
$$\vspace{0.005mm}

\noindent by induction.\bigskip

(iii) Let finally $Y$ be upper unitriangular. It is easily checked
that $Y$ can be written as a finite product of matrices of the
form
   \be \label{elemform} \lp\begin{array}{cc} \I & \frak{E} \\  & \I\end{array}\rp\,,
   \qquad
   \lp\begin{array}{cc} \frak{U} &  \\ & \I\end{array}\rp
   \quad
   \text{and}
   \quad
   \lp\begin{array}{cc} \I & \\  & \frak{U}\end{array}\rp\;, \ee
where $\frak{U}$ is upper unitriangular and $\frak{E}$ is
elementary. It thus suffices to consider a matrix $Y$ of each one
of the preceding ``elementary forms''. Moreover, it also suffices
to prove multiplicativity for a lower unitriangular $X$.\medskip

Set
$$X=\lp\begin{array}{cc}\frak{X}_1&\\\frak{X}_3&\frak{X}_2\end{array}\rp\;,$$
where $\frak{X}_1,\frak{X}_2$ are lower unitriangular.\medskip

(a) If $Y$ is of the first above elementary form,
$$XY=\lp\begin{array}{cc}\frak{X_1}&\frak{X_1E}\\[4pt]
\frak{X_3}&\frak{X_2+X_3E}\end{array}\rp\;.$$ Hence, using the
induction and Lemma \ref{AB BA}, we get
$$
\begin{array}{rcl}
\gdet(XY)&=&
\gdet(\frak{X_2+X_3E})\cdot\gdet(\frak{X_1-X_1E(X_2+X_3E)^{-1}X_3})\\[4pt]
&=&\gdet(\frak{X_1})\cdot\gdet(\frak{X_2+X_3E})
\cdot\gdet(\frak{\I-E(X_2+X_3E)^{-1}X_3})\\[4pt]
&=&\gdet(\frak{X_1})\cdot\gdet(\frak{X_2+X_3E})
\cdot\gdet(\frak{\I-(X_2+X_3E)^{-1}X_3E})\\[4pt]
&=&\gdet(X)\cdot\gdet(Y)\,.
\end{array}
$$

(b) If $Y$ is of the second elementary form,
$$XY=\lp\begin{array}{cc}\frak{X_1U}& \\ \frak{X_3U}& \frak{X_2}\end{array}\rp\;,$$
then
$$\gdet(XY)=\gdet(\frak{X_2})\cdot\gdet(\frak{X_1})\cdot\gdet(\frak{U})
=\gdet(X)\cdot\gdet(Y)\;.$$

(c) If $Y$ is of the last form, the proof of multiplicativity is
analogous to that in (b).\medskip

This completes the proof of multiplicativity and thus of the
formal existence and uniqueness of the graded determinant.

\subsection{Polynomial Structure}

\subsubsection{Quasideterminants and Homological Relations}

Let as above $X\in\textsf{M}^0(\mathbf{r}_0;A)$, let all the
components of $\mathbf{r}_0$ be $1$ (or 0), and set
$r_0=|\,\mathbf{r}_0|$. We will need the following lemma.

\begin{lem} \label{Lemma1}
For $r\neq i$ and $s \neq j$, we have \bea  \label{HomRel+}
\bigl|X\bigr|_{ij} \bigl|X^{i,j}\bigr|_{rl}  =  \pm\;
\bigl|X\bigr|_{il} \bigl|X^{i,l}\bigr|_{rj} &\mbox{ and }&
\bigl|X\bigr|_{ij} \bigl|X^{i,j}\bigr|_{ks}  =  \pm\;
\bigl|X\bigr|_{kj} \bigl|X^{k,j}\bigr|_{is}\;. \eea
\end{lem}

\begin{proof}
The result is a consequence of an equivalent definition of
quasideterminants and of the \textit{homological relations}
\cite{GGRW05}. More precisely, the quasideterminant $|{X}|_{ij}$
can be defined by $|{X}|_{ij}=({X}^{-1})_{ji}^{-1}$. It follows
that Definition \ref{defQuasidet} of quasideterminants reads
$$
\left|X\right|_{ij}= x_{ij}-\sum_{\substack{a\neq i \\ b\neq j}}
x_{ib}\left|X^{i,j}\right|_{ab}^{-1} x_{aj}\;.
$$
An induction on the matrix dimension then shows that any
quasideterminant $|X|_{ij}$ is homogeneous. Hence, the mentioned
homological relations \bea \bigl|X\bigr|_{ij}
\bigl|X^{i,l}\bigr|_{rj}^{-1} = -\bigl|X\bigr|_{il}
\bigl|X^{i,j}\bigr|_{rl}^{-1} & \mbox{ and } &
\bigl|X^{k,j}\bigr|_{is}^{-1} \bigl|X\bigr|_{ij} =
-\bigl|X^{i,j}\bigr|_{ks}^{-1} \bigl|X\bigr|_{kj}\;, \eea valid
for $r\neq i, s\neq j$, are equivalent to (\ref{HomRel+}).
\end{proof}

\begin{prop}\label{Arbitrary}
Set $$D(X)=|X|_{11}|X^{1,1}|_{22}\cdot\ldots\cdot x_{r_0{}r_0}$$
and let $(i_1 , \ldots , i_{r_0})$, $(j_1, \ldots ,j_{r_0})$ be
two permutations of $(1,\ldots ,r_0)$. Then,
$$ D(X)= \pm\; |X|_{i_1j_1}|X^{i_1,j_1}|_{i_2j_2}\cdot\ldots\cdot x_{i_{r_0}j_{r_0}}=\pm\; |X|_{i_1j_1}D(X^{i_1,j_1})\;.$$
\end{prop}

    \begin{proof} It suffices to use Lemma \ref{Lemma1}.
    \end{proof}

\begin{prop}\label{PolyDTheo}
The product $D(X)$ is linear with respect to the rows and columns
of $X$.
\end{prop}

    \begin{proof}
For $r_0=1$, the claim is obvious. For $r_0=2$, we obtain
$$ D(X)=D\!\lp \begin{array}{cc} x_{11} & x_{12} \\ x_{21} & x_{22} \end{array}\rp = \lp x_{11}-x_{12}x_{22}^{-1}x_{21} \rp  x_{22}= x_{11}x_{22} \pm x_{12}x_{21}\;.  $$

Assume now that the statement holds up to $r_0=n$ ($n\ge 2$) and
consider the case $r_0=n+1$. We have $$ D(X)= \left| X
\right|_{11} D\!\lp X^{1,1}\rp $$ and
$$  \left|X\right|_{11}= x_{11}-\sum_{\substack{a\neq 1 \\ b\neq 1}} x_{1b}\left|X^{1,1}\right|_{ab}^{-1} x_{a1}\;,  $$
with
$$ \left|X^{1,1}\right|_{ab}^{-1}= \pm\; D\!\lp X^{1a,1b} \rp D^{-1}\! \lp X^{1,1} \rp\;,$$
due to Proposition \ref{Arbitrary}. Therefore,
$$ D(X)= x_{11} D\!\lp X^{1,1} \rp \pm \sum_{\substack{a\neq 1 \\ b\neq 1}} x_{1b}D\!\lp X^{1a,1b} \rp x_{a1}\;. $$
By induction, the products $D\!\lp X^{1,1}\rp$ and $D\!\lp
X^{1a,1b} \rp$ are linear with respect to the rows and columns of
their arguments. Hence, the result.
    \end{proof}

\subsubsection{Preliminary Remarks}\medskip

(i) Consider a block matrix $$W=\left(\begin{array}{ccc}A & 0 & B \\
C & D & E \\ F & 0 & G\end{array}\right)$$ with square diagonal
blocks over a unital ring. A straightforward computation allows to
check that the formal inverse of $W$ is given by
\be\label{FormalInverse}W^{-1}= \left(\begin{array}{ccc}A' & 0 &
B'
\\[6pt] -D^{-1}(C A'+ E F') & D^{-1} & -D^{-1}(C B'+ E G')\\[6pt]
 F' & 0 &G'
 \end{array}\right)\,,\ee where $$\left(\begin{array}{cc}A'& B'\\[4pt]
F' & G'\end{array}\right):=\left(\begin{array}{cc}A & B \\[4pt] F &
G\end{array}\right)^{-1}\,.$$\medskip

(ii) Let $X\in\sss{M}^0(\mathbf{r}_0;A)$, denote by
$E_{\za\zb}(\zl)\in \sss{M}^0(\mathbf{r}_0;A)$ the elementary
matrix whose unique nonzero element $\zl\in A^{w_{\za}+w_{\zb}}$
is located at position $(\za,\zb)$, $\za\neq \zb$ (remember that
any row index $\za$ defines a unique block row index $k$ and
therefore a unique degree $w_{\za}:=\zg_k$), and set
$$
G_{\za\zb}(\zl):=\I+E_{\za\zb}(\zl)\in
\sss{M}^0(\mathbf{r}_0;A)\,.
$$
The rows of the product matrix $$X_{\op{red}}:=G_{\za\zb}(\zl)X\in
\sss{M}^0(\mathbf{r}_0;A)$$ (for an explanation of the notation $X_{\op{red}}$, see (for instance) Equation (\ref{MotivRed})) are the same than those of $X$, except
that its $\za$-th row is the sum of the $\za$-th row of $X$ and of
the $\zb$-th row of $X$ left-multiplied by $\zl$. Since $\gdet$ is
formally multiplicative and as it immediately follows from its
definition that $\gdet\lp G_{\za\zb}(\zl) \rp =1$, we get \be
\gdet(X)=\gdet(X_{\op{red}})\;. \label{GDetRed}\ee

(iii) In the following we write $X^{i:j}$, if we consider the
matrix obtained from $X$ by deletion of its $i$-th row
$(x_{i1},x_{i2}\ldots)$ and $j$-th column, whereas in $X^{i,j}$
the superscripts refer, as elsewhere in this text, to a block row
and column. Subscripts characterizing quasideterminants should be
understood with respect to the block decomposition.

\begin{lem}\label{x11Lem} Let $$X=\left(\begin{array}{c|c}x_{11} & \star \cdots  \star\\\hline 0&  \\ \vdots & X^{1:1}  \\ 0  &\end{array}\right)\in
\sss{M}^0(\mathbf{r}_0+\mathbf{e}_1;A)\,,$$ where
$\mathbf{e}_1=(1,0,\ldots, 0)$. Then,
$$\gdet(X)=x_{11}\gdet(X^{1:1})\,.$$\end{lem}

\begin{proof} If the first component of $\mathbf{r}_0\in \N^q$ vanishes, the result is obvious. Otherwise, it suffices to remember
that $(|X|_{11})^{1:1}=|X^{1:1}|_{11}$, so that $$
|X|_{11}=\left(\begin{array}{c|c}x_{11} & * \cdots *\\\hline 0&
\\ \vdots & |X^{1:1}|_{11}  \\ 0 &\end{array}\right)\,.$$
Therefore,
$$\gdet(X)=\det \left|X\right|_{11} \gdet\lp X^{1,1}\rp= x_{11}
\det\left|X^{1:1}\right|_{11} \gdet\lp (X^{1:1})^{1,1}\rp = x_{11}
\gdet(X^{1:1})\,.$$ Hence the lemma.\end{proof}

\subsubsection{Proof of the Polynomial Character}

We are now ready to give the proof of Theorem \ref{thmgdet}, Part
(ii).

Fix $n$, so $q=2^{n-1}$ is fixed as well.
We first consider the case $\mathbf{r}_0\in\{0,1\}^{\times q}.$ Since the
quasideterminants in Definition (\ref{gdetExpl}) of $\gdet$ are
then quasideterminants over $A$ valued in $A^0$, we have
$\gdet(X)=D(X)$, we conclude that $\gdet(X)$ is linear
in the rows and columns of $X$, due to Proposition
\ref{PolyDTheo}.\medskip

To prove that $\gdet(X)$, where $X\in\sss{M}^0(\mathbf{r}_0;A)$, is
linear for any $\mathbf{r}_0\in\N^q$, it suffices to show that, if
$\gdet$ is linear for
$\mathbf{r}_0=(r_1,\ldots,r_{q})\in\{0,1,\ldots,R\}^{\times q}$,
$R\ge 1$, then it is linear as well for
$\mathbf{r}_0+\mathbf{e}_{\ell}=(r_1,\ldots,r_{\ell}+1,\ldots,r_{q})$,
with $r_{\ell}\neq 0$.\medskip

(i) We just mentioned that linearity of $\gdet$ with respect to
the rows and columns of its argument holds true for $R=1$.\medskip

(ii) Suppose now that it is valid for some $\mathbf{r}_0$ and some
$R\ge 1$.\medskip

(ii1) We first prove that linearity then still holds for
$\mathbf{r}_0+\mathbf{e}_1$. More precisely, we set $|\mathbf{r}|=r_1+\ldots
+r_{q}$, consider a matrix

\be X=\lp \begin{array}{cccc}
            x_{11} & x_{12} & \ldots & x_{1,N+1} \\
            \vdots & \vdots &        & \vdots \\
            x_{k1} & x_{k2} & \ldots & x_{k,N+1} \\[4pt]
            0       & x_{k+1,2}& \ldots & x_{k+1,N+1}\\
            \vdots & \vdots &           & \vdots \\
            0 & x_{N+1,2} & \ldots & x_{N+1,N+1}
         \end{array} \rp\in\sss{M}^0(\mathbf{r}_0+\mathbf{e}_1;A)\,,
         \label{kform}\ee\vspace{3mm}

\noindent and prove linearity of $\gdet(X)$ by induction on $k$.
To differentiate the two mentioned inductions, we speak about the
induction in $|\mathbf{r}|$ and in $k$.\medskip

(a) For $k=1$, Lemma \ref{x11Lem} yields
$\gdet(X)= x_{11} \gdet(X^{1:1})$ and the $|\mathbf{r}|$-induction assumption allows to conclude that $\gdet(X)$ is linear.\medskip

(b) If $k=2$, consider $G_{21}(-x_{21}x_{11}^{-1})$, where
$\zl=-x_{21}x_{11}^{-1}$ has the same degree as $x_{21}$. Matrix
\be\label{MotivRed}X_{\op{red}}=G_{21}(-x_{21}x_{11}^{-1})X\ee has the form
(\ref{kform}) with $k=1$. Indeed, its rows are those of $X$,
except for the second one, which reads
$$0\;,\;x_{22}-x_{21}x_{11}^{-1}x_{12}\;,\;x_{23}-x_{21}x_{11}^{-1}x_{13}\;,\;\ldots$$
Hence, by (\ref{GDetRed}) and (a), $$\gdet(X)=\gdet(X_{\op{red}})
=x_{11}\gdet(X_{\op{red}}^{1:1})$$ is linear in the rows and
columns of $X_{\op{red}}$. This means in fact that $\gdet(X)$ is
linear with respect to the rows and columns of $X$.\medskip

(c) Assume now that $\gdet(X)$, $X$ of the form (\ref{kform}), is
linear up to $k=\zk\ge 2$ and examine the case $k=\zk+1$. We use
the same idea as in (b), but have now at least two possibilities.
The matrix
$$X_{\op{red}}^1=G_{\zk+1,1}(-x_{\zk+1,1}x_{11}^{-1})\,X\quad
(\text{ resp., }\;
X_{\op{red}}^2=G_{\zk+1,2}(-x_{\zk+1,1}x_{21}^{-1})\,X\,)$$ has
the form (\ref{kform}) with $k=\zk$ and, in view of the
$k$-induction assumption, $\gdet(X)=\gdet(X_{\op{red}}^1)$ is
linear in the rows and columns of $X_{\op{red}}^1$ and thus
contains at the worst $x_{11}^{-1}$ (resp.,
$\gdet(X)=\gdet(X_{\op{red}}^2)$ is linear and contains at the
worst $x_{21}^{-1}$). It follows e.g., that
$$x_{21}\gdet(X_{\op{red}}^1)=x_{21}\gdet(X_{\op{red}}^2)$$ is
polynomial in the entries of $X$, so that $\gdet(X_{\op{red}}^1)$
cannot contain $x_{11}^{-1}$. Therefore,
$\gdet(X)=\gdet(X_{\op{red}}^1)$ is linear in the rows and columns
of $X\in\sss{M}^0(\mathbf{r}_0+\mathbf{e}_1;A)$.\medskip

(ii2) The case
$X\in\sss{M}^0(\mathbf{r}_0+\mathbf{e}_{\ell};A)$, where $\ell\neq 1$,
can be studied in a quite similar way. Indeed, consider first a
matrix of the form

$$ X=\lp \begin{array}{ccc|cccc}
            x_{11}     & \ldots & x_{1,m-1}    & 0       & x_{1,m+1} &\ldots & x_{1,N+1} \\
            \vdots      & &   \vdots       & \vdots & \vdots       &      & \vdots \\
            x_{m-1,1}  & \ldots & x_{m-1,m-1}& 0       & x_{m-1,m+1} &\ldots &
            x_{m-1,N+1}\\[4pt]\hline
            x_{m1}     & \ldots & x_{m,m-1}    & x_{mm} & x_{m,m+1} &\ldots & x_{m,N+1}\\
            x_{m+1,1}  & \ldots & x_{m+1,m-1}& 0       & x_{m+1,m+1} &\ldots & x_{m+1,N+1}\\
             \vdots     &        &  \vdots        &  \vdots & \vdots      &       & \vdots       \\
            x_{N+1,1} & \ldots & x_{N+1,m-1} & 0    & x_{N+1,m+1} & \ldots & x_{N+1,N+1}
         \end{array}
         \rp\in\sss{M}^0(\mathbf{r}_0+\mathbf{e}_{\ell};A)\,,$$

\vspace{5mm}\noindent where the suggested redecomposition
corresponds to the redecomposition
$$(r_1,\ldots,r_{\ell-1}\;|\;r_{\ell}+1,\ldots,r_{q})\,.$$ Remember
that the determinant of $X$ is defined as
$$\gdet(X)=\prod_{i=0}^{q-1}\det\left|X^{1\ldots i,1\ldots
i}\right|_{i+1\;i+1}\,.$$ It follows from (an obvious extension of)
Lemma \ref{x11Lem} that
$$\prod_{i=\ell}^{q-1}\det\left|X^{1\ldots i,1\ldots
i}\right|_{i+1\;i+1}=x_{mm}\prod_{i=\ell}^{q-1}\det\left|(X^{m:m})^{1\ldots
i,1\ldots i}\right|_{i+1\;i+1}\,.$$\noindent Let now $i<\ell$ and
let
$$ \left|X^{1\ldots i,1\ldots i}\right|_{i+1\;i+1} = X_{i+1\;i+1}-UW^{-1}V$$
be the corresponding quasideterminant. Since the $m$-th column of
$U$ vanishes (we maintain the numeration of $X$), $$ UW^{-1} =
U^{0:m}(W^{-1})^{m:0}\;,$$ where $0$ means that no row or column
has been deleted. Using Equation (\ref{FormalInverse}), we
similarly find that
$$ UW^{-1}V = U^{0:{m}}(W^{-1})^{m:m}V^{{m}:0}\;=U^{0:{m}}(W^{m:m})^{-1}V^{{m}:0}\,.$$
Hence, $$\left|X^{1\ldots i,1\ldots
i}\right|_{i+1\;i+1}=\left|(X^{m:m})^{1\ldots i,1\ldots
i}\right|_{i+1\;i+1}$$ and $$ \gdet(X)=x_{mm}
\gdet(X^{{m}:{m}})\;.$$ Linearity now follows from the
$|\mathbf{r}|$-induction hypothesis. To pass from an elementary
$m$-th column, containing a unique nonzero element, to an
arbitrary one, it suffices to ``fill'' the elementary column
downwards and upwards using the arguments detailed in (b) and (c)
of (ii1). This then completes the proof of the polynomial
structure of the $\zG$-determinant, as well as that of Theorem
\ref{thmgdet}.

\subsection{Example}
The graded determinant of a matrix

$$X=\lp \begin{array}{c|c|c|c} x&a&b&c \\[4pt]
\hline d&y&e&f \\ [4pt]
\hline g&h&z&l \\ [4pt]
\hline m&n&p&w  \end{array} \rp\in\sss{M}^0\left((1,1,1,1);A\right)$$

\vspace{3mm}\noindent over a $(\Z_2)^3$-graded commutative algebra
$A$, is given by

$$\gdet(X)= \left| X \right|_{11} \left| X^{1,1} \right|_{22}
\left| X^{12,12} \right|_{33} \left| X^{123,123} \right|_{44}\,.$$

Of course, $$ \left| X^{123,123} \right|_{44} = w\quad \text{ and
}\quad \left| X^{12,12} \right|_{33} = z-l\za p\,,$$

\vspace{3mm}\noindent where $\za=w^{-1}.$ Using Inversion Formula
(\ref{InvForm}) and the graded commutativity of the multiplication
in $A$, we easily find

$$\left| X^{1,1} \right|_{22} = \za \zb \lp y(zx-lp) -ehw+fph+eln - fnz
\rp\,,$$

\vspace{3mm}\noindent with $\zb=(z-l\za p)^{-1}$, so that the
product of the three last factors of $\gdet(X)$ is equal to
$$v:= y(zx-lp) -ehw+fph+eln - fnz\,.$$

As concerns the quasideterminant $|X|_{11}$, the inverse of the
involved $(3\times 3)$-matrix can be computed for instance by means
of the UDL-decomposition of this matrix. After simplifications
based on graded commutativity, we obtain

$$ \lp\begin{array}{ccc}
y & e&f \\[4pt]
 h&z&l \\[4pt]
  n&p&w
  \end{array}\rp^{-1} =
   \lp\begin{array}{ccc} v^{-1}(zw-lp) & v^{-1}(fp-ew) & v^{-1}(el-fz) \\[4pt]
v^{-1}(ln-hw)&v^{-1}(yw-fn)&v^{-1}(hf-ly) \\[4pt]
v^{-1}(ph-zn)&v^{-1}(ne-py)&v^{-1}(yz-eh)
\end{array} \rp\,. $$
Finally,
\beas \left| X \right|_{11} &=& v^{-1}\left[ x v - \lp a(zw-lp)+b (ln-hw)+c(ph-zn)  \rp d \right.\\
                             &&       - \lp a(fp-ew)+b (yw-fn)+c (ne-py) \rp g \\
                            & &       \left. - \lp a(el-fz)+b (hf-ly)+c(yz-eh)  \rp m \right]\; \eeas
and
\begin{equation}
 \label{poly2}
\begin{array}{rcrcrcrcrcrcl}
   \gdet(X)&=& xyzw&-&xylp&-&xehw&-&xfhp&+&xeln&-&xfzn\\[4pt]
               &&         -adzw&+&adlp&+&aegw&+&afgp&-&aelm&+&afzm \\[4pt]
               &&         -bdhw&+&bdln&-&bygw&+&bfgn&+&bylm&+&bfhm \\[4pt]
               &&       -cdhp&-&cdzn&-&cygp&+&cegn&-&cyzm&+&cehm\,.
\end{array}
\end{equation}

Further examples are given in Section \ref{ExSect}.

\begin{rem}
(a) As claimed by Theorem \ref{thmgdet}, Part (ii),
the determinant $\gdet(X)$ is linear in the rows and columns of $X$. Result (\ref{poly2}) is thus analogous
to the Leibniz formula for the classical determinant. Of course,
signs are quintessentially different. It is worth noticing that,
if we use the LDU-decomposition of $X$, which can be obtained in
the exact same manner as the UDL-decomposition, see Proposition
\ref{UDL}, we find that \be \label{gdetExpl2} \gdet(X)=
\prod_{k=1}^{q}\det | X^{k+1\ldots q\, , \,k+1\ldots q}|_{kk}\,,
\ee where $q=2^{n-1}.$ For the preceding example, we thus get

$$\gdet(X)=|X^{234,234}|_{11}\,|X^{34,34}|_{22}\,|X^{4,4}|_{33}\,|X|_{44}\,$$

\vspace{3mm}\noindent and computations along the same lines as
above actually lead exactly to Expression (\ref{poly2}).\medskip

(b)
The reader might wish to check by direct inspection that the polynomials
$\gdet(XY)$ and $\gdet(X)\cdot\gdet(Y)$ coincide.
However, even the simplest example in the $(\Z_2)^3$-graded case
involves over a hundred of terms. Such (computer-based) tests preceded
the elaboration of our above proofs. These computations can of
course not be reproduced here.\end{rem}





\section{$(\Z_2)^n$-Graded Berezinian of Invertible Graded Matrices of Degree $0$}\label{Ber}

Let $A$ be a $(\Z_2)^n$-graded commutative algebra. Its even part
$A_0$ is clearly $(\Z_2)^n_{\;0}$-graded commutative. In the
preceding section, we investigated the determinant $\gdet(X)$ of
degree zero $(\Z_2)^n_{\;0}$-graded matrices
$X\in\sss{M}^0(\mathbf{r}_0;A_0)$, with $\mathbf{r}_0 \in \N ^q$,
$q := 2^{n-1}$.
Below we now define the determinant $\gber(X)$ of invertible
degree zero $(\Z_2)^n$-graded matrices
$X\in\sss{GL}^{0}(\mathbf{r}; A)$, with $\mathbf{r} \in \N^p$,
$p:=2^n$. These matrices, which contain $p \times p$ blocks
$X_{ku}$ of dimension $r_k\times r_u$ with entries in
$A^{\zg_k+\zg_u}$, can also be viewed as $2 \times2$ block
matrices $\X$, whose decomposition is given by the even and odd
subsets of $(\Z_2)^n$ (let us emphasize that the present $2\times 2$ decomposition ${\X}$ of a graded matrix $X$ is different from the above used $2\times 2$ decomposition ${\frak X}$ (see Equation (\ref{FirstDecomposition})) of a purely even graded matrix $X$).

\subsection{Statement of the Fundamental Theorem}

Recall that, $(A^{0})^{\times}$ denotes the
group of units of the unital algebra $A^0.$
Clearly, $(A^0)^{\times}=\sss{M}^0((1,0, \ldots , 0) ;A).$
Our result is as follows.

\begin{thm}

\label{ExUniGber} There is a unique group homomorphism
$$ \gber: \sss{GL}^{0}(\mathbf{r}; A) \to (A^{0})^{\times} $$
such that:
\begin{enumerate}
        \item[($1$)] For every $2\times 2$ block-diagonal matrix $\mathcal{X}\in\sss{GL}^{0}(\mathbf{r}; A)$,
        $$ \gber(\mathcal{X})=
        \gdet(\mathcal{X}_{11})\cdot\gdetinv(\mathcal{X}_{22})\in(A^0)^{\times}\,.$$
    \item[($2$)] The image of any lower $($resp., upper$)$ $2\times 2$ block-unitriangular matrix in $\sss{GL}^{0}(\mathbf{r}; A)$ equals $1\in(A^0)^{\times}$.
\end{enumerate}
\end{thm}

We call $\gber(X)$, where $X \in \sss{GL}^{0}(\mathbf{r}; A)$, the
\textit{graded Berezinian} of $X$.\medskip

To prove the theorem, we will need the following lemma.

\begin{lem} \label{invlem}
A homogeneous degree $0$ matrix  $X \in \sss{M}^{0}(\mathbf{r}; A)$ is
invertible, i.e., $X\in \sss{GL}^{0}(\mathbf{r}; A)$, if and only
if its diagonal blocks $\X _{11}$ and $\X _{22}$ are invertible.
\end{lem}

\begin{proof}
Let $J$ be the ideal of $A$ generated by the odd elements, i.e., the elements of the subspace
$$
A_1:=\bigoplus_{\zg\in(\Z_2)^n_1}A^{\zg}\,.
$$
Note that any element in $J$ reads as finite sum
$\sum_{k=1}^s a_ko_k$, where $a_k\in A$ and $o_k\in A_1$ are
homogeneous. Therefore, we have $j^{s+1}=0$, so that $1\notin J$
and $J$ is proper. We denote by $\bar{A}$ the associative unital
quotient algebra $\EnsQuot{A}{J}\neq \{\bar 0\}$ and by $\,\bar{} : A \to
\bar{A}$ the canonical surjective algebra morphism. This map
induces a map on matrices over $A$, which sends a matrix $Y$ with
entries $y_{ij}\in A$ to the matrix $\bar{Y}$ with entries
$\bar{y}_{ij}\in \bar{A}$.\medskip

It suffices to prove the claim: \emph{A matrix $Y$ over $A$ is invertible if and
only if $\bar{Y}$ over $\bar A$ is invertible}. Indeed, if $X$ is a degree
$0$ graded matrix over $A$, its $2\times 2$ blocks $\X_{12}$ and
$\X_{21}$ have exclusively odd entries. Thus, $X$ is invertible if
and only if $$ \bar{\X} = \lp \begin{array}{cc} \bar{\X}_{11} &
\bar{0}
\\[4pt] \bar{0} & \bar{\X}_{22} \end{array} \rp $$ is, i.e., if and
only if $\bar{\X}_{11}$ and $\bar{\X}_{22}$ are invertible, or
better still, if and only if $\X_{11}$ and $\X_{22}$ are
invertible.\medskip

As for the mentioned claim, if $Y$ is invertible, then, clearly,
$\bar{Y}$ is invertible as well. Conversely, assume $\bar{Y}$
invertible and focus for instance on the right inverses (arguments are the same
for the left ones). There then exists a matrix $Z$ over $A$ such
that $YZ=\I+W$, for some matrix $W$ over $J$. Hence, matrix $Y$ has a right inverse, if $\I+W$ is invertible, which happens if $W$ is
nilpotent.
Note that $W^{S+1}=0$, then
$$
(\I+W)^{-1}=\I+\sum_{k=1}^S(-W)^k\,.
$$
To see that matrix $W$ over
$J$ is actually nilpotent, remark that there is a finite number of
homogeneous odd elements $o_1, \ldots ,o_S$ such that each entry
of $W$ reads $\sum_{k=1}^S a_k o_k$, with homogeneous $a_k \in A$.
Hence, $W^{S+1}=0.$\end{proof}

\subsection{Explicit Expression}

As for the graded determinant, we will prove uniqueness and
existence of the graded Berezinian by giving a necessary explicit
formula and then proving that a homomorphism defined by means of
this formula fulfills all the conditions of Theorem
\ref{ExUniGber}.

\begin{prop}
\label{BerDetPr}
Let $A$ be a $(\Z_2)^n$-graded commutative algebra,
and let $\mathbf{r}\in\N^{2^n}$.
The $(\Z_2)^n$-graded Berezinian of a matrix $X\in\sss{GL}^{0}(\mathbf{r}; A)$ is given by
\be\gber(X)=
\gdet(|\X|_{11})\cdot\gdetinv(\X_{22})
\,,\label{DefGBer2}\ee
where $\X_{k\ell}$ refers to the $2\times 2$ redivision of $X$ (see beginning of Section \ref{Ber}).
\end{prop}

Of course, for $n=1$, we recover the classical Berezinian $\ber$.\medskip

We will use the following lemma.

\begin{lem}\label{gdetgtr}
If $X ,Y \in \sss{M}^0(\mathbf{r}; A)$ and $\X _{12}$ or $\Y
_{21}$ is elementary, then
$$\gdet\lp  \I -\X _{12}\Y _{21}\rp= \gdet\lp \I + \Y_{21} \X _{12} \rp\,.$$
\end{lem}

\begin{proof}
Note first that both sides of the result actually make sense,
since $ \X _{12}\Y _{21} $ and $\Y _{21} \X _{12}$ are
$(\Z_2)^n_{\;0}$-graded matrices of degree $0$. The proof is analogous to that of
Lemma \ref{AB BA}; the sign change is due to oddness of the
entries of the off-diagonal blocks of $\X$ and $\Y$.
\end{proof}

\begin{proof}[Proof of Proposition \ref{BerDetPr} and Theorem \ref{ExUniGber}]
In view of Lemma \ref{invlem}, any matrix
$X\in\sss{GL}^{0}(\mathbf{r}; A)$ admits a $2\times 2$ UDL
decomposition. The conditions of Theorem \ref{ExUniGber} then
obviously imply that if $\gber$ exists it is necessarily given by
Equation (\ref{DefGBer2}).

To prove existence of the graded Berezinian map, it suffices to
check that the map $\gber$ defined by Equation (\ref{DefGBer2}) is
multiplicative and satisfies the requirements ($1$) and ($2$). Properties ($1$) and ($2$) are obvious. As for multiplicativity,
let $X , Y \in \sss{GL}^{0}(\mathbf{r}; A)$ and consider their
$2\times 2$ UDL decomposition
$$
\X =\X_U\X_D\X_L=\lp \begin{array}{cc} \I & \X_{12}\X_{22}^{-1}\\ 0 & \I \end{array} \rp \lp \begin{array}{cc} \left| \X\right| _{11} &0 \\ 0&\X_{22} \end{array} \rp \lp \begin{array}{cc} \I &0 \\ \X_{22}^{-1}\X_{21} & \I \end{array} \rp
$$
and
$$
\Y =\Y_U\Y_D\Y_L=\lp \begin{array}{cc} \I & \Y_{12}\Y_{22}^{-1}\\ 0 & \I \end{array} \rp \lp \begin{array}{cc} \left| \Y\right| _{11} &0 \\ 0&\Y_{22} \end{array} \rp \lp \begin{array}{cc} \I &0 \\ \Y_{22}^{-1}\Y_{21} & \I \end{array}
 \rp\,.
$$
The product $XY$ then reads
\be\label{MultiplicProd} \X\Y= \X_U
\lp
\begin{array}{cc} \left| \X\right|
_{11} \left| \Y\right| _{11} &\left| \X\right| _{11} \Y_{12} \\[6pt]
\X_{21}\left| \Y\right| _{11}& \X_{21}\Y_{12}+\X_{22}\Y_{22}
\end{array}\rp \Y_L\,. \ee
Comparing now the $2\times 2$ UDL decomposition of
$XY\in\sss{GL}^{0}(\mathbf{r}; A)$ given by the formula used above
for $\X$ and $\Y$, to that obtained via the $2\times 2$ UDL
decomposition of the central factor of the {\small RHS} of Equation
(\ref{MultiplicProd}) (which exists e.g., as
$\X_{21}\Y_{12}+\X_{22}\Y_{22}=(\mathcal{XY})_{22}$ is
invertible), we find in particular that
$$\left| \mathcal{XY} \right|_{11}= \left| \mathcal{X}\right| _{11}\left| \mathcal{Y}\right| _{11}-
\left| \X\right|
_{11}\Y_{12}\lp\X_{21}\Y_{12}+\X_{22}\Y_{22}\rp^{-1}\X_{21}\left|
\Y\right| _{11}\,.
$$
Consequently, \be \label{gberex}\gber({XY})=\ee $$\gdet\lp \left|
\X\right| _{11}\left| \Y\right| _{11}-  \left| \X\right|
_{11}\Y_{12}\lp\X_{21}\Y_{12}+\X_{22}\Y_{22}\rp^{-1}\X_{21}\left|
\Y\right| _{11} \rp\cdot \gdetinv\lp
\X_{21}\Y_{12}+\X_{22}\Y_{22}\rp\,.$$

\noindent Clearly, if $X$ is $2\times 2$ upper unitriangular or
diagonal, or $Y$ is lower unitriangular or diagonal (i.e., $X=\X_U$
or $X=\X_D$, $Y=\Y_L$ or $Y=\Y_D$), Equation (\ref{gberex})
reduces to
\be\label{MultiplicBer}\gber(XY)=\gber(X)\cdot\gber(Y)\,,\ee in
view of the multiplicativity of $\gdet$.\medskip

Due to the just established right multiplicativity of $\gber$ for matrices of the type $\Y_L$ and $\Y_D$, it suffices to still prove
right multiplicativity for matrices of the type $\Y_U$. Since any matrix $\Y_U$ reads as a finite product of matrices of the form $$\E_U:=\lp\begin{array}{cc} \I & E \\ 0 & \I
\end{array}\rp\in\sss{GL}^{0}(\mathbf{r}; A)\;,$$ where $E$ is elementary, we only need prove this right multiplicativity for $\E_U$. However, since we know already that the graded Berezinian is left multiplicative for matrices of the form $\X_U$ and $\X_D$, we can even confine ourselves to showing that Equation (\ref{MultiplicBer}) holds true for $$X=\X_L=\lp
\begin{array}{cc} \I & 0 \\ C & \I \end{array}\rp$$ and $Y=\E_U$, i.e., to showing that $\gber(\X_L\E_U)=1$.\medskip

By definition,
$$ \gber({\X_L\E_U})= \gdet (\I - E\lp \I+CE  \rp^{-1}C )\cdot \gdetinv( \I+CE )\,. $$

\noindent It is easily checked that any entry of $CE$ and $EC$
vanishes or is a multiple of the unique nonzero element of $E$.
Since this element is odd, it squares to zero and
$(CE)^{2}=(EC)^{2}=0$. This implies in particular that $\lp\I +CE
\rp^{-1}=\I-CE$, so that $$ \gber({\X_L\E_U})= \gdet\lp  \I - EC
\rp \cdot\gdetinv\lp  \I+CE \rp\,. $$ When combining the latter
result with the consequence \beas \gdet(\I - EC) = \gdet(\I +
CE)\eeas of Lemma \ref{gdetgtr}, we eventually get $
\gber({\X_L\E_U})=1.$ This completes the proofs of Proposition
\ref{BerDetPr} and Theorem \ref{ExUniGber}.
\end{proof}

\section{The Liouville Formula}

In this section we explain the relation between the graded trace
and the graded Berezinian.

\subsection{Classical Liouville Formulas}

The well-known Liouville formula \be\label{LiouvilleClass}
\det(\exp(X))=\exp(\tr(X))\,, \ee $X\in\op{gl}(r,\C)$, expresses
the fact that the determinant is the group analog of the trace. A
similar statement
$$
\ber(\I+\ze X)=1+\ze\op{str}(X)\,,
$$
where $\ze$ is an even parameter such that $\ze^2=0$ and $X$ an
even matrix, holds true in Superalgebra, see \cite{Lei80}, \cite{DM99}, \cite{CCF11}. Moreover, if $A$ denotes the Grassmann
algebra generated over a commutative field $\K$, where $\K=\R$ or
$\K=\C$, by a finite number of anticommuting parameters
$\xi_1,\ldots,\xi_q$, i.e., if $A=\K[\xi_1,\ldots,\xi_q]$, we have
$$\frac{d}{dt}\ber(X)=\op{str}(M)\ber(X)\;,$$ if
$\frac{d}{dt}X=MX$, where $t$ is a real variable and $X=X(t)$
(resp., $M$) is an invertible even (resp., an even) matrix over
$A$, see \cite{Ber87}. The super counterpart
$$\ber(\exp(X)):=\exp(\op{str}(X))\,,$$ is valid if $X$ is (even) nilpotent and in contexts where the exponential series converge, see \cite{DM99}. For another
variant of the correspondence between the Berezinian and the
supertrace, using a formal parameter, see \cite{Man88}.\medskip

\subsection{Graded Liouville Formula}

In this section, we extend the preceding relationship to $\gber$
and $\gtr$. We introduce a formal nilpotent degree 0 parameter
$\ze$ and work over the set $A[\ze]$ of formal polynomials in
$\ze$ with coefficients in a $(\Z_2)^n$-graded commutative algebra
$A$ over reals. It is easily seen that $A[\ze]$ is itself a
$(\Z_2)^n$-graded commutative algebra, so that the graded Berezinian
and graded trace do exist over $A[\ze]$.

For any
$\mathbf{s}\in\N^p$ and any $M\in\sss{M}^0(\mathbf{s};A)$, we
have $\ze M\in\sss{M}^0(\mathbf{s};A[\ze])$ and we define
$$
\exp{(\ze M)}:=\sum_k\ze^k\frac{M^k}{k!}\in\sss{M}^0(\mathbf{s};A[\ze])\,.
$$
Moreover, if $t$ denotes a real variable, it is straightforward
that
\be\label{DerExp}\frac{d}{dt}\exp(t\,\ze M)
=(\ze M)\exp(t\,\ze
M)\;.\ee In particular, if $X\in\sss{GL}^{0}(\mathbf{r}; A)$, we get
$$\ze\gtr({X})=\gtr(\ze X)\in
A^0[\ze]=\sss{M}^0((1,0,\ldots,0);A[\ze])\,,$$  so that
$$\exp(\gtr(\ze X))\in (A^0[\ze])^{\times}\quad\text{and}\quad
\gber(\exp(\ze X))\in (A^0[\ze])^{\times}\,.$$

\begin{thm}\label{thmLiouville}
If $\ze$ denotes a formal nilpotent parameter of degree $0$ and
$X$ a graded matrix of degree 0, we have $$ \gber(\exp(\ze X))=
\exp \lp \gtr(\ze X) \rp\,.$$
\end{thm}

Some preliminary results are needed to prove the preceding
theorem. Let $X=X(t),M=M(t)\!\in \sss{M}^0(\mathbf{r}_0;A)$, where
$t$ runs through an open internal $I\subset \R$, let $A$ be
finite-dimensional, and assume that the dependence of $X=X(t)$ on
$t$ is differentiable.

\begin{lem} If $\frac{d}{dt}X =MX$,
then
$$
\frac{d}{dt}\gdet(X)=\gtr(M)\gdet(X)\,.
$$
\label{DiffEqGDet}
\end{lem}

\begin{proof}
Denote by $|\mathbf{r}|=r_1+\ldots +r_q$ the total matrix
dimension and set $X=(x_{ij})$, $M=(m_{ij})$. As $\gdet(X)$ is
linear in the rows and columns of $X$, we have $$\frac{d}{dt} \gdet(X)
= \sum_{i=1}^{|\mathbf{r}|} \gdet(X_{i\,'})\,,$$ where $X_{i\,'}$ is the matrix
$X$ with $i$-th row derived with respect to $t$.
For any fixed $i$
and any $j \neq i$, the matrix $G_{ij}(-m_{ij})\in
\sss{M}^0(\mathbf{r}_0;A)$, see Proof of Theorem \ref{thmgdet}, Part (ii),
has graded determinant 1. Hence, the graded
determinant of $G_{ij}(-m_{ij})\,X_{i\,'}$, $j\neq i$, coincides
with that of $X_{i\,'}$, although we subtract from the $i$-th row
of $X_{i\,'}$ its $j$-th row left-multiplied by $m_{ij}$. In view
of the assumption in Lemma \ref{DiffEqGDet}, $$\frac{d}{dt} x_{ia}=
\sum_j m_{ij}x_{ja}\,,$$ for all $a$. Consequently, the $i$-th row
of $\prod_{j\neq i}G_{ij}(-m_{ij})\,X_{i\,'}$ contains the
elements
$$\frac{d}{dt}x_{ia}-\sum_{j\neq i}m_{ij}x_{ja}=m_{ii}x_{ia}\,,\quad
a\in\{1,\ldots,N\}\;$$ and $$\prod_{j\neq
i}G_{ij}(-m_{ij})\,X_{i\,'}=\op{diag}(1,\ldots,1,m_{ii},1,\ldots,1)\,X\,,$$
with self-explaining notation. It follows that
$$\frac{d}{dt} \gdet(X)
= \sum_{i=1}^{|\mathbf{r}|} \gdet(\prod_{j\neq
i}G_{ij}(-m_{ij})\,X_{i\,'})=\sum_{i=1}^{|\mathbf{r}|}
m_{ii}\gdet(X)=\gtr(M)\gdet(X)\,.$$\end{proof}

Take now $X=X(t),M=M(t)\!\in \sss{GL}^{0}(\mathbf{r}; A)$, $t\in
I\subset \R$, where $X=X(t)$ depends again differentiably on $t$
and is invertible for any $t$.

\begin{lem}\label{DiffEqGBer}
If $\frac{d}{dt}X=MX$, then
$$\frac{d}{dt}\gber(X)=\gtr(M)\gber(X)\,.$$
\end{lem}

The proof is exactly the same as for the classical Berezinian
\cite{Ber87}. We reproduce it here to ensure independent
readability of this text.

\begin{proof} Set $Y:=|\X|_{11}$ and $Z:=\X_{22}^{-1}$. A short computation shows
that
$$\frac{d}{dt}Y=(\mathcal{M}_{11}-\X_{12}\X_{22}^{-1}\mathcal{M}_{21})\,Y=:P\,Y
\quad\text{ and }\quad
\frac{d}{dt}Z=-Z\,(\mathcal{M}_{21}\X_{12}\X_{22}^{-1}+\mathcal{M}_{22})=:-Z\,Q\,.$$
Since, from Lemma \ref{DiffEqGDet}, we now get
$$\frac{d}{dt}\gdet(Y)=\gtr(P)\,\gdet(Y)\quad\text{ and }\quad
\frac{d}{dt}\gdet(Z)=-\gtr(Q)\,\gdet(Z)\,,$$ we obtain
$$\frac{d}{dt}\gber(X)=\lp\frac{d}{dt}\gdet(Y)\rp\cdot\gdet(Z)+\gdet(Y)\cdot\lp\frac{d}{dt}\gdet(Z)\rp$$
$$=\Big(\gtr(P)-\gtr(Q)\Big)\,\gdet(Y)\cdot\gdet(Z)\,.$$ It follows
from the Lie algebra homomorphism property of $\gtr$ (see Theorem
\ref{TThmGen}) that, for every ${X}\in\sss{M}^{0}(\mathbf{r}; A)$,
one has
$$
\gtr({\mathcal{X}_{12}\mathcal{X}_{21}})=
-\gtr({\mathcal{X}_{21}\mathcal{X}_{12}})\,.\label{gtrLahomExt}
$$
Using the preceding equation and recalling that above
$\mathcal{M}_{22}$ is viewed as a purely even degree 0 matrix, we
see that
$$\gtr(P)-\gtr(Q)=\gtr(\mathcal{M}_{11})-\gtr(\mathcal{M}_{22})=\gtr(M)\,.$$
Hence Lemma \ref{DiffEqGBer}.
\end{proof}

\begin{proof}[Proof of Theorem \ref{thmLiouville}]
 It follows from Equation (\ref{DerExp}) that $$\frac{d}{dt}\exp(t\,\ze X)=(\ze X)\exp(t\,\ze X)\quad\text{ and }\quad
\frac{d}{dt}\exp(t\gtr(\ze X))=\gtr(\ze X)\exp(t\gtr(\ze X))\,.$$
However, $$\frac{d}{dt}\gber(\exp(t\,\ze X))=\gtr(\ze
X)\gber(\exp(t\,\ze X))\,,$$ due to Lemma \ref{DiffEqGBer}. It now
suffices to observe that both solutions $\exp(t\gtr(\ze X))$ and
$\gber(\exp(t$ $\ze X))$ of the equation $\frac{d}{dt} y=\gtr(\ze
X)\,y$ coincide at $0$.
\end{proof}

\section{$(\Z_2)^n$-Graded Determinant over Quaternions and Clifford Algebras}

It this section we obtain the results specific for the Clifford Algebras
and, in particular, for the algebra of quaternions.
We show that, in the quaternionic case, the graded determinant is related to
the classical Dieudonn\'e determinant.
We then examine whether the graded determinant can be extended to
(purely even) homogeneous matrices of degree $\zg\neq0$.
It turns out that this is possible under the condition that the global dimension
is equal to $0$ or $1$ modulo $4$.\medskip

\subsection{Relation to the Dieudonn\'e Determinant}

In this section the algebra $A$
is the classical algebra $\mathbb{H}$ of quaternions
equipped with the $(\Z_2)^3_{\;0}$-grading (see Example \ref{quatgrad}).
It turns out that the graded determinant of a purely even
homogeneous quaternionic matrix of degree $0$ coincides
(up to a sign) with the classical Dieudonn\'e determinant.

\begin{prop}
For any matrix $X\in\sss{M}^0(\mathbf{r}_0;\qH)$ of degree $0$,
the graded determinant is a real number and its absolute value coincides
with the Dieudonn\'e determinant:
$$
\left|\gdet(X)\right|=\Ddet(X)\,.
$$
\end{prop}

\begin{proof}
We will first show that the graded determinant
can be written as a product of quasiminors
$$
\gdet(X)=|X|_{i_1j_1}|X^{i_1:j_1}|_{i_2j_2}\ldots
x_{i_Nj_N}\,,
$$
for appropriate permutations $I=(i_1,\ldots,i_N),
J=(j_1,\ldots,j_N)$ of $(1,\ldots, N)$,
$$|\mathbf{r}|=|\mathbf{r}_0|=r_1+\cdots+r_4\,.$$
and then compare this formula with the classical Dieudonn\'e determinant.\medskip

Following \cite{GRW03}, \cite{GGRW05}, define the {\it predeterminants}
of $X$ by
$$
D_{IJ}(X):=
|X|_{i_1j_1}|X^{i_1:j_1}|_{i_2j_2}\,\ldots\, x_{i_Ni_N}\;\in \qH
$$
where $I=(i_1,\ldots,i_N)$ and $J=(j_1,\ldots,j_N)$
are some  are permutations of $(1,\ldots,N)$.
It is shown in the above references, that the $D_{IJ}(X)$
are polynomial expressions with real coefficients in the entries $x_{ij}$
and their conjugates $\overline{x}_{ij}$.
Moreover, for any of these permutations $I,J$,
the {\it Dieudonn\'e determinant} $\Ddet(X)$ of
$X$ is given by
\be
\label{DieuDet}
\Ddet(X)=||D_{IJ}(X)||\,,
\ee
where $||-||$ denotes the quaternionic norm.

Observe that in our case, $X\in\sss{GL}^0(\mathbf{r}_0;\qH)$, (i.e., $X$
is an invertible $4\times 4$ block matrix with square diagonal blocks,
such that the entries of block $X_{ku}$ are elements of $\qH^{\zg_k+\zg_u}$),
the entries $x_{ij}$ of $X$ and their conjugates coincide (up to sign).
Hence, every $D_{IJ}(X)$ is polynomial in the entries $x_{ij}$.
Moreover, these polynomials are clearly valued in $\R$.
Therefore,
\be \label{DieuDet2}\Ddet(X)=
|D_{IJ}(X)|\,,\ee
where $|-|$ is the absolute value of real numbers. \medskip

For any matrix $X\in\sss{M}^0(\mathbf{r}_0;\qH)$, we obtain the
graded determinant of $X$ by writing the rational
expression
\be\label{QuatGradDet} \det\left|X
\right|_{11}\det\left|X^{1,1} \right|_{22}\det\left|X^{12,12}
\right|_{33}\det X_{44}\;\ee
which is, indeed, a polynomial (see Theorem \ref{thmgdet}, Part (ii)). Let us recall that,
for a matrix $C$ with commutative entries, a quasideterminant is a
ratio of classical determinants \cite{GGRW05}: $$\det
C=(-1)^{a+b}|C|_{ab}\det(C^{a:b})\,.$$ When applying this result
iteratively to the determinants in (\ref{QuatGradDet}), we get $$\det|X^{1\ldots k,1\ldots
k}|_{k+1\;k+1}=\pm\prod_{i=0}^{r_{k+1}-1}\left|\lp|X^{1\ldots
k,1\ldots k}|_{k+1\;k+1}\rp^{1\ldots i:1\ldots
i}\right|_{i+1\;i+1}\,.$$ Corollary \ref{HPAugmented} now entails
that the rational expression (\ref{QuatGradDet}) coincides with $$\pm
|X|_{11}|X^{1:1}|_{22}\ldots x_{NN}=\pm\, D(X)=\pm
|X|_{i_1j_1}|X^{i_1:j_1}|_{i_2j_2}\ldots
x_{i_Nj_N}=\pm\,D_{IJ}(X)\,,$$ see Proposition \ref{Arbitrary},
for any permutations $I=(i_1,\ldots,i_N), J=(j_1,\ldots,j_N)$ of
$(1,\ldots, N)$. However, see Equation (\ref{DieuDet2}), for
appropriate permutations $I,J$, the product $\pm\,D_{IJ}(X)\in\R$
is polynomial and thus coincides with $\gdet(X)$.
\end{proof}

\subsection{Graded Determinant of Even Homogeneous Matrices of Arbitrary Degree}

In this section, $A$ denotes a $(\Z_2)^n_0$-graded commutative
associative unital algebra, such that each subspace
$A^{\zg}$ contains at least one invertible element.
Every Clifford algebra  satisfies the required property
since it is graded division algebra, see Section~\ref{CliS}.\medskip

Consider a homogeneous matrix
$X\in\sss{M}^{\zg}(\mathbf{r}_0;A),$
where $\zg\in(\Z_2)^n_0$ is not necessarily equal to $0$.
Every such matrix can be written (in many different ways) in the form
$X= q\, X_0$, where $X_0$ is homogeneous of degree $0$ and $q \in A$ is invertible.
We define the {\it graded determinant} of $X$
by
\be\label{gdetNonZero}
\gdet(X):= q^{|\mathbf{r}|} \gdet(X_0)\ee
with values in $A^{|\mathbf{r}|\,\zg}$.

Let us first check that the graded determinant is \textit{well-defined}.
Given two invertible elements  $q, q'\in{}A^{\zg}$, one has two different expressions:
$X= (q\I) X_0  = (q'\I) X_0'\,.$
Since
$$ X_0= (q^{-1}\I) (q'\I) X_0'= \lp q^{-1}q'\I\rp X_0'\,, $$
where both factors of the {\small RHS} are of degree 0, we obtain
$$
q^{|\mathbf{r}|}\gdet(X_0)
= q^{|\mathbf{r}|}(q^{-1}q')^{|\mathbf{r}|}\gdet(X_0')=
q'^{|\mathbf{r}|}\gdet(X_0')\,.
$$
Therefore, Formula (\ref{gdetNonZero}) is independent of the choice elements $q$.

\begin{prop}
\label{01Thm}
 The graded determinant (\ref{gdetNonZero}) is multiplicative:
$$\gdet(XY)=\gdet(X)\cdot\gdet(Y)\,.$$
 for any purely even homogeneous $(|\mathbf{r}|\times |\mathbf{r}|)$-matrices $X,Y$,
 if and only if $|\mathbf{r}|=0,1$ $(\op{mod} 4)$.
\end{prop}

\begin{proof}
Recall that the $A$-module structure (\ref{AmoduleStruc}) of the space
$\sss{M}(\mathbf{r};A)$ is compatible
with the associative algebra structure in the sense that, for any
$a,b\in A$, and matrices $X\in\sss{M}^{x}(\mathbf{r};A)$, and
$Y\in\sss{M}(\mathbf{r};A)$, we have
$$(aX)(bY)=(-1)^{\la\tilde b,x\ra}(ab)(XY)\,.$$

Let $X$ and $Y$ be two purely even graded matrices
of even degree $\zg_{\ell}$ and $\zg_m$, respectively. We then
have
$$
\gdet(X)\gdet(Y)=q_{\ell}^{|\mathbf{r}|}q_m^{|\mathbf{r}|}\gdet(X_0)\gdet(Y_0)
$$
and, since $XY=(q_{\ell}q_m\I)(X_0Y_0)$, we
get
$$\gdet(XY)=(q_{\ell}q_m)^{|\mathbf{r}|}\gdet(X_0)\gdet(Y_0)=
(-1)^{\tiny\frac{|\mathbf{r}|(|\mathbf{r}|-1)}{2}}\,
q_{\ell}^{|\mathbf{r}|}q_m^{|\mathbf{r}|}\,\gdet(X_0)\gdet(Y_0)\,.$$
Therefore, multiplicativity is equivalent to the condition
$(-1)^{\tiny\frac{|\mathbf{r}|(|\mathbf{r}|-1)}{2}}=1$,  that holds if and only if
$|\mathbf{r}|=0,1$ $(\op{mod} 4)$.
\end{proof}

\begin{rem}
 It is
well-known that the classical super determinant can be extended to
odd matrices, only if the numbers $p$ of even and $q$ of odd
dimensions coincide, hence only if the total dimension $|\mathbf{r}|=p+q=0$
$\op{(mod\; 2)}$. Although our situation is not completely
analogous, this can explain that a condition on the total dimension shows up
in our situation.
\end{rem}

\section{Examples of Quaternionic $(\Z_2)^n$-Graded Determinants}\label{ExSect}

In this last section we present several examples of matrices, their
traces and determinants, in the $(\Z_2)^3$-graded case. A natural
source of such matrices is provided by endomorphisms of modules
over the classical algebra $\mathbb{H}$ of quaternions
equipped with the $(\Z_2)^3_0$-grading (\ref{DegH}). Although this section
is based on the general theory developed in the present work, it can be
read independently and might provide some insight into the more abstract aspects of this text.

\subsection{\label{QuatZero}Quaternionic Matrices of Degree Zero}

The examples given in this section are obtained by straightforward
computations that we omit.

\subsubsection{Matrix Dimension $|\mathbf{r}|=4=1+1+1+1$}\label{Ex1Sec}

The first interesting case of $(\Z_2)^3$-graded matrices is that
of dimension $|\mathbf{r}|=4$. More precisely, let $V$ be a real
$4$-dimensional vector space, graded by the even
elements of $(\Z_2)^3\,$:
\begin{equation}
\label{SpaceV}
V=V_{(0,0,0)} \oplus V_{(0,1,1)} \oplus V_{(1,0,1)}
\oplus V_{(1,1,0)}\,.
\end{equation}
Each of the preceding subspaces is $1$-dimensional. We then define a
$(\Z_2)^3$-graded $\mathbb{H}$-module $M=V\otimes_\R\mathbb{H}$. A homogeneous degree $(0,0,0)$ endomorphism of $M$ is then
represented by a matrix of the form
$$
X= \left(
 \begin{array}{l|l|l|l}
x&a\,\qi&b\,\qj&c\,\qk\\[6pt]
 \hline
d\,\qi\Top&y&e\,\qk&f\,\qj\\[6pt]
 \hline
g\,\qj\Top&h\,\qk&z&\ell\,\qi\\[6pt]
 \hline
m\,\qk\Top&n\,\qj&p\,\qi&w
 \end{array}
 \right)\,,
$$

\vspace{2mm}\noindent where the coefficients $x,a,\ldots,w$ are
real numbers and where $\qi,\qj,\qk\in\mathbb{H}$ stand for the
standard basic quaternions.

\label{Ex1111} The graded trace of $X$ is, in the considered situation of a purely even grading and a degree (0,0,0) matrix, just the usual trace $\gtr(X)=x+y+z+w\;.$ The graded determinant is given by
\begin{equation}
\label{ExDet1}
\begin{array}{rcrcrcrcrcrcl}
\G\!\!\det\left(X\right)&=&
        xyzw   & + & xy\ell p & + & xehw  & + & xfhp & -  & xe\ell n  & + & xfzn\\[4pt]
& &   adzw & + & ad\ell p & + & aegw & + & afgp & + & ae\ell m & -  & afzm\\[4pt]
& & -bdhw & + & bd\ell n & + & bygw & + & bfgn  & + & by\ell m & + & bfhm\\[4pt]
& &   cdhp & + & cdzn      & - & cygp   & + & cegn & + & cyzm    &
+ & cehm\,.
\end{array}
\end{equation}

\vspace{2mm}

 The signs look at first sight quite
surprising. However, in this quaternionic degree $0$ case,
$$\left|\,\G\!\!\det\left(X\right)\right|=\mathrm{D}\!\det\left(X\right)\,,$$ where $\Ddet$
denotes the Dieudonn\'e determinant.
Note also that (\ref{ExDet1}) is a particular case of formula (\ref{poly2}).

\subsubsection{Matrix Dimension $|\mathbf{r}|=4=0+2+1+1$}\label{Ex2Sec}

When choosing other dimensions for the homogeneous subspaces of the
$4$-dimensional real vector space $V$, see
(\ref{SpaceV}), namely
$$
V_{(0,0,0)}=0\,, \qquad \dim V_{(0,1,1)}=2\,, \qquad
\dim V_{(1,0,1)}=\dim V_{(1,1,0)}=1\,,
$$
we obtain a different type of matrix representation of degree $(0,0,0)$
endomorphisms of the $\qH$-module $M=V\otimes_{\R}\qH$:
$$
X= \left(
 \begin{array}{ll|l|l}
x&a&b\,\qk&c\,\qj\\[6pt]
d&y&e\,\qk&f\,\qj\\[6pt]
 \hline
g\,\qk\Top&h\,\qk&z&\ell\,\qi\\[6pt]
 \hline
m\,\qj\Top&n\,\qj&p\,\qi&w
 \end{array}
 \right)\,.
$$

\vspace{2mm}The graded determinant of $X$ is then given by
$$
\begin{array}{rcrcrcrcrcrcl}
\G\!\!\det\left(X\right)&=&
        xyzw   & + & xy\ell p & + & xehw  & + & xfhp & -  & xe\ell n  & + & xfzn\\[4pt]
& &  - adzw & - & ad\ell p & - & aegw & - & afgp & + & ae\ell m & -  & afzm\\[4pt]
& & -bdhw & + & bd\ell n & + & bygw & + & bfgn  & - & by\ell m & - & bfhm\\[4pt]
& &  -cdhp & - & cdzn      & + & cygp   & - & cegn & + & cyzm    &
+ & cehm\, .
\end{array}
$$
The signs are of course different from those in
(\ref{ExDet1}). The graded determinant is multiplicative, i.e.,
$$\gdet (XY)=\gdet\left(X\right)\cdot \gdet\left(Y\right)\,,
$$ (this property can be checked by direct computation) and it satisfies the Liouville formula
$$\gdet\left(\exp(\ze X)\right)=\exp\left(\gtr(\ze X)\right)\,,$$ where $\ze$ denotes a degree zero nilpotent parameter.

\subsubsection{Matrix Dimension $|\mathbf{r}|=d+d+d+d$}

In this example, the graded components of the space (\ref{SpaceV})
are of equal dimension $d$. Then, there exists an embedding of the
quaternion algebra $\mathbb{H}$ into the algebra of quaternionic
matrices of homogeneous degree $(0,0,0)$. Indeed, consider
$q=x+a\qi+b\qj+c\qk$ and set
$$
X_q= \left(
 \begin{array}{l|l|l|l}
x\Top&a\qi&b\qj&c\qk\\[6pt]
\hline
a\qi\Top&x&c\qk&b\qj\\[6pt]
 \hline
b\qj\Top&c\qk&x&a\qi\\[6pt]
 \hline
c\qk\Top&b\qj&a\qi&x
 \end{array}
 \right)\,,
 $$
\vspace{2mm}\noindent where the blocks are $(d\times{}d)$-matrices proportional to the
identity.\medskip

In this case, the graded determinant is
$
\gdet(X_q)=||q||^{2d}\,.
$

\subsection{\label{uatNonZero} Homogeneous Quaternionic Matrices of Nonzero Degrees}

In this last subsection, the $(\Z_2)^3$-graded space (\ref{SpaceV}) is of dimension
$$
|\mathbf{r}|=r_1+r_2+r_3+r_4\,,
\qquad
|\mathbf{r}|=0,1\;(\text{mod } 4 )\,,
$$
where the $r_i$ are the dimensions of the four homogeneous subspaces.
Let us emphasize that the condition $|\mathbf{r}|=0,1\;(\text{mod }4 )$ is necessary and sufficient
for consistency.\medskip

\subsubsection{Multiplying by a Scalar.}

If $q$ denotes a nonzero homogeneous quaternion (i.e., it is a nonzero multiple of an element of the standard basis of $\qH$) and if $X$ is a quaternionic matrix of degree $(0,0,0)$,
then
$$
\gdet(q\,X)=q^{|\mathbf{r}|}\,\gdet(X)\,.
$$
Since every homogeneous quaternionic matrix, of any even degree, is of the form $q\,X$, this definition allows
to calculate the determinants from the results of Subsection \ref{QuatZero}.\medskip

Let us emphasize that the multiplication of a graded matrix $X$ by a homogeneous scalar $q$ obeys a nontrivial sign rule.\medskip

(a) For instance, in the case of the decomposition $(1,1,1,1)$, one has

$$\label{SignRule1}
\qi\left(
\begin{array}{c|c|c|c}
1&&&\\
\hline
&1&&\\
\hline
&&1&\\
\hline &&&1
\end{array}
\right)= \left(
\begin{array}{c|c|c|c}
\qi&&&\\
\hline
&\,\qi\,&&\\
\hline
&&\!\!-\qi\!\!&\\
\hline &&&\!\!-\qi\!
\end{array}
\right)\,, \qquad \qj\left(
\begin{array}{c|c|c|c}
1&&&\\
\hline
&1&&\\
\hline
&&1&\\
\hline &&&1
\end{array}
\right)= \left(
\begin{array}{c|c|c|c}
\qj&&&\\
\hline
&\!\!-\qj\!\!&&\\
\hline
&&\,\qj\,&\\
\hline &&&\!\!-\qj\!
\end{array}
\right)\,,
$$
and similarly for $\qk$, with $-$ signs at the second and the third blocks.

(b) For the decomposition $(0,2,1,1)$,

$$\label{SingRule2}\qj\left(
\begin{array}{cc|c|c}
1&&&\\
&1&&\\
\hline
&&1&\\
\hline &&&1
\end{array}
\right)= \left(
\begin{array}{cc|c|c}
\!\!-\qj\!\!&&&\\
&\!\!-\qj\!\!&&\\
\hline
&&\,\qj\,&\\
\hline &&&\!\!-\qj\!
\end{array}
\right)\,,\quad \qk\left(
\begin{array}{cc|c|c}
1&&&\\
&1&&\\
\hline
&&1&\\
\hline &&&1
\end{array}
\right)= \left(
\begin{array}{cc|c|c}
-\qk&&&\\
&\,-\qk\,&&\\
\hline
&&\!\!-\qk\!\!&\\
\hline &&&\qk\!
\end{array}
\right)\,,
$$
and similarly for $\qi$,  with $-$ signs at the third and the fourth blocks.

\subsubsection{An Example in Dimension $|\mathbf{r}|=1+1+2+1$}

Consider the example
$$\qi\,\I=\qi\,\left(
\begin{array}{c|c|cc|c}
1&&&&\\\hline
&1&&&\\
\hline &&1&&\\&&&1&\\ \hline &&&&1
\end{array}
\right)= \left(
\begin{array}{c|c|cc|c}
\qi&&&&\\\hline
&\qi&&&\\
\hline &&-\qi&&\\&&&-\qi&\\ \hline &&&&-\qi
\end{array}
\right)\; \in\sss{M}^{(011)}((1,1,2,1);\qH)\,.$$
According to the definition (\ref{gdetNonZero}), one has: $\gdet(\qi\I)=\qi^5=\qi$.
Applying (heuristically)
Liouville's formula, we find
$$
\label{gdetNonZero02}\gdet(q\,\I)=\gdet(q\,\I_N)=\exp\lp\gtr\lp
\qi{\tiny\frac{\zp}{2}}\;\I_N\rp\rp=\exp\lp
\qi{\tiny}\frac{5\zp}{2}\rp=\qi\,,
$$
in full accordance with the definition.

\subsubsection{The Diagonal Subalgebra $\mathbb{H}$}

The diagonal $(|\mathbf{r}|\times{}|\mathbf{r}|)$-matrices
\begin{equation}
\label{IJK} I= \left(
\begin{array}{rcl}
\qi&&\\
&\ddots&\\
&&\qi
\end{array}
\right)\,, \qquad J= \left(
\begin{array}{rcl}
\qj&&\\
&\ddots&\\
&&\qj
\end{array}
\right)\,, \qquad K= \left(
\begin{array}{rcl}
\qk&&\\
&\ddots&\\
&&\qk
\end{array}
\right)
\end{equation}
are homogeneous of degree $(0,1,1),\;(1,0,1),\;(1,1,0)$,
respectively. These matrices $I,J,K$, together with the identity
matrix, span a subalgebra of the algebra of quaternionic graded matrices, which is isomorphic to $\mathbb{H}$.\medskip

\label{TRIJK} (a) For the matrices \eqref{IJK}, the graded trace is

$$\begin{array}{rcl}
\gtr(I)&=&\left(r_1+r_2-r_3-r_4\right)\qi \,,\\[6pt]
\gtr(J)&=&\left(r_1-r_2+r_3-r_4\right)\qj \,,\\[6pt]
\gtr(K)&=&\left(r_1-r_2-r_3+r_4\right)\qk \,.
\end{array}
$$

\noindent\vspace{3mm} In particular, for the decomposition $|\mathbf{r}|=1+1+1+1$, one obtains
$$
\gtr\left(I\right)=\gtr\left(J\right)=\gtr\left(K\right)=0\,,
$$
while for $|\mathbf{r}|=0+2+1+1$,
$$
\gtr\left(I\right)=0\,,
\qquad \gtr\left(J\right)=-2\qj\,,
\qquad \gtr\left(K\right)=-2\qk\,.
$$

(b) The graded determinant of the matrices \eqref{IJK} is as
follows:
$$
\gdet\left(I\right)=\qi^{\left(r_1+r_2-r_3-r_4\right)}\,, \qquad
\gdet\left(J\right)=\qj^{\left(r_1-r_2+r_3-r_4\right)}\,, \qquad
\gdet\left(K\right)=\qk^{\left(r_1-r_2-r_3+r_4\right)}\,.
$$

For example, if $|\mathbf{r}|=1+1+1+1$, one has
$$\G\!\!\det\left(I\right)=\G\!\!\det\left(J\right)=\G\!\!\det\left(K\right)=1\,.$$
If $|\mathbf{r}|=0+2+1+1$, then
$$
\gdet\left(I\right)=1\,, \qquad
\gdet\left(J\right)=\gdet\left(K\right)=-1\,.
$$

\medskip
\medskip

{\bf Acknowledgments}.
We are pleased to thank Sophie Morier-Genoud, Christian Duval, Dimitri Gurevich and
Hovannes Khudaverdian
for enlightening discussions and Dimitry Leites for valuable comments (in particular, to have made us aware of Nekludova's work).
T.C. thanks the
Luxembourgian NRF for support via AFR grant 2010-1 786207.
The research of N.P. was supported by Grant
GeoAlgPhys 2011-2013 awarded by the University of Luxembourg.

\bigskip


\end{document}